\documentclass[11pt]{amsart}  
\usepackage[english]{babel} 

\usepackage{dsfont}\let\mathbb\mathds 
  
\usepackage{times}       
 
\usepackage{mathrsfs}        

\usepackage{hyperref}  
 
\numberwithin{equation}{section} 

\usepackage{pgf, tikz}
\usetikzlibrary{matrix,positioning}
\usetikzlibrary{arrows,decorations.markings}

\newtheorem*{theoremnn}{Theorem}   
\newtheorem{theorem}{Theorem}[section]

\newtheorem{definition}[theorem]{Definition}
\newtheorem{proposition}[theorem]{Proposition}
\newtheorem{lemma}[theorem]{Lemma}
\newtheorem{corollary}[theorem]{Corollary}

\newtheorem{remark}[theorem]{Remark}

\addto\captionsfrenchb{}
\addto\captionsfrenchb{}

    \newlength{\myarrowsize} 
    \newlength{\myoldlinewidth}

    \pgfarrowsdeclare{myto}{myto}{
        \pgfsetdash{}{0pt} 
        \pgfsetbeveljoin 
        \pgfsetroundcap 
        \setlength{\myarrowsize}{0.5pt}
        \addtolength{\myarrowsize}{.5\pgflinewidth}
        \pgfarrowsleftextend{-4\myarrowsize-.5\pgflinewidth} 
        \pgfarrowsrightextend{.7\pgflinewidth}
    }{
        \setlength{\myarrowsize}{0.4pt} 
        \addtolength{\myarrowsize}{.3\pgflinewidth}  
        \setlength{\myoldlinewidth}{\pgflinewidth}
        \pgfsetroundjoin
        \pgfsetlinewidth{0.0001pt}
        \pgfpathmoveto{\pgfpoint{0.43\myarrowsize}{0}}
        \pgfpatharc{0}{70}{0.14\myarrowsize}
        \pgfpatharc{-80}{-169.5}{4\myarrowsize}
        \pgfpatharc{150}{189}{0.95\myarrowsize and 0.95\myarrowsize}
        \pgfpatharc{0}{-40}{15\myarrowsize}
        \pgfpathmoveto{\pgfpoint{0.43\myarrowsize}{0}}
        \pgfpatharc{0}{-70}{0.14\myarrowsize}
        \pgfpatharc{80}{169.5}{4\myarrowsize}
        \pgfpatharc{-150}{-189}{0.95\myarrowsize and 0.95\myarrowsize}
        \pgfpatharc{0}{40}{15\myarrowsize}
        \pgfpathclose
        \pgfsetstrokeopacity{0.25}
        \pgfusepathqfillstroke
    }

    \pgfarrowsdeclare{myonto}{myonto}{
        \pgfsetdash{}{0pt} 
        \pgfsetbeveljoin 
        \pgfsetroundcap 
        \setlength{\myarrowsize}{0.5pt}
        \addtolength{\myarrowsize}{.5\pgflinewidth}
        \pgfarrowsleftextend{-4\myarrowsize-.5\pgflinewidth} 
        \pgfarrowsrightextend{.7\pgflinewidth}
    }{
        \setlength{\myarrowsize}{0.4pt} 
        \addtolength{\myarrowsize}{.3\pgflinewidth}  
        \setlength{\myoldlinewidth}{\pgflinewidth}
        \pgfsetroundjoin
        \pgfsetlinewidth{0.0001pt}
        \pgfpathmoveto{\pgfpoint{0.43\myarrowsize}{0}}
        \pgfpatharc{0}{70}{0.14\myarrowsize}
        \pgfpatharc{-80}{-169.5}{4\myarrowsize}
        \pgfpatharc{150}{189}{0.95\myarrowsize and 0.95\myarrowsize}
        \pgfpatharc{0}{-40}{15\myarrowsize}
        \pgfpathmoveto{\pgfpoint{-7\myarrowsize}{0}}
				\pgfpatharc{0}{70}{0.14\myarrowsize}
        \pgfpatharc{-80}{-169.5}{4\myarrowsize}
        \pgfpatharc{150}{189}{0.95\myarrowsize and 0.95\myarrowsize}
        \pgfpatharc{0}{-40}{15\myarrowsize}
        \pgfpathmoveto{\pgfpoint{0.43\myarrowsize}{0}}
        \pgfpatharc{0}{-70}{0.14\myarrowsize}
        \pgfpatharc{80}{169.5}{4\myarrowsize}
        \pgfpatharc{-150}{-189}{0.95\myarrowsize and 0.95\myarrowsize}
        \pgfpatharc{0}{40}{15\myarrowsize}
			  \pgfpathmoveto{\pgfpoint{-7\myarrowsize}{0}}
        \pgfpatharc{0}{-70}{0.14\myarrowsize}
        \pgfpatharc{80}{169.5}{4\myarrowsize}
        \pgfpatharc{-150}{-189}{0.95\myarrowsize and 0.95\myarrowsize}
        \pgfpatharc{0}{40}{15\myarrowsize}
        \pgfpathclose
        \pgfsetstrokeopacity{0.25}
        \pgfusepathqfillstroke
        \pgfusepathqfillstroke
    }

 \pgfarrowsdeclare{myhook}{myhook}{
        \setlength{\myarrowsize}{0.6pt}
        \addtolength{\myarrowsize}{.5\pgflinewidth}
        \pgfarrowsleftextend{-4\myarrowsize-.5\pgflinewidth} 
        \pgfarrowsrightextend{.7\pgflinewidth}
    }{
        \setlength{\myarrowsize}{0.6pt} 
        \addtolength{\myarrowsize}{.5\pgflinewidth}  
        \pgfsetdash{}{+0pt}
        \pgfsetroundcap
        \pgfpathmoveto{\pgfqpoint{-2pt}{-6\pgflinewidth}}
        \pgfpathcurveto
            {\pgfqpoint{4\pgflinewidth}{-4.667\pgflinewidth}}
            {\pgfqpoint{4\pgflinewidth}{0pt}}
            {\pgfpointorigin}
        \pgfusepathqstroke
    }

		\pgfarrowsdeclare{my to}{my to}
{
  \pgfarrowsleftextend{-2\pgflinewidth}
  \pgfarrowsrightextend{\pgflinewidth}
}
{
  \pgfsetlinewidth{0.8\pgflinewidth}
  \pgfsetdash{}{0pt}
  \pgfsetroundcap
  \pgfsetroundjoin
  \pgfpathmoveto{\pgfpoint{-5.5\pgflinewidth}{7.5\pgflinewidth}}
  \pgfpathcurveto
  {\pgfpoint{-4.0\pgflinewidth}{0.1\pgflinewidth}}
  {\pgfpoint{0pt}{0.25\pgflinewidth}}
  {\pgfpoint{0.75\pgflinewidth}{0pt}}
  \pgfpathcurveto
  {\pgfpoint{0pt}{-0.25\pgflinewidth}}
  {\pgfpoint{-4.0\pgflinewidth}{-0.1\pgflinewidth}}
  {\pgfpoint{-5.5\pgflinewidth}{-7.5\pgflinewidth}}
  \pgfusepathqstroke
}

		\pgfarrowsdeclare{mytodashed}{mytodashed}
{
 \pgfarrowsleftextend{-2\pgflinewidth}
 \pgfarrowsrightextend{\pgflinewidth}
}
{
 \pgfsetlinewidth{0.8\pgflinewidth}
  \pgfsetdash{{0.8cm}{0.8cm}}{0cm}
\pgfsetroundcap
\pgfsetroundjoin
  \pgfpathmoveto{\pgfpoint{-5.5\pgflinewidth}{7.5\pgflinewidth}}
\pgfpathcurveto
{\pgfpoint{-4.0\pgflinewidth}{0.1\pgflinewidth}}
 {\pgfpoint{0pt}{0.25\pgflinewidth}}
{\pgfpoint{0.75\pgflinewidth}{0pt}}
\pgfpathcurveto
{\pgfpoint{0pt}{-0.25\pgflinewidth}}
  {\pgfpoint{-4.0\pgflinewidth}{-0.1\pgflinewidth}}
 {\pgfpoint{-5.5\pgflinewidth}{-7.5\pgflinewidth}}
\pgfusepathqstroke
}

\tikzstyle{vecArrow} = [thick, decoration={markings,mark=at position
   1 with {\arrow[semithick]{open triangle 60}}},
   double distance=1.4pt, shorten >= 5.5pt,
   preaction = {decorate},
   postaction = {draw,line width=1.4pt, white,shorten >= 4.5pt}]
\tikzstyle{innerWhite} = [semithick, white,line width=1.4pt, shorten >= 4.5pt]

	\makeatletter
	\newcommand\POSITION[3]{%
	\begingroup
	\@tempdim@x=0cm
	\@tempdim@y=\paperheight
	\advance\@tempdim@x#1
	\advance\@tempdim@y-#2
	\put(\LenToUnit{\@tempdim@x},\LenToUnit{\@tempdim@y}){#3}%
	\endgroup
	}
\makeatother

\addto\captionsenglish{}

\begin{document}

 \title[Canonical form in $\tilde{A}_n$ ]{Canonical reduced expression in  affine Coxeter groups\\
 Part I - Type $\tilde{A}_n$  }  
\author{Sadek AL HARBAT} 
 
	\begin{abstract} We classify the elements of $W(\tilde{A}_n)$ by giving a  canonical reduced expression for each,  using basic tools among which affine length. We give some direct consequences for such a canonical form: a description of left multiplication by a simple reflection, a study  of the right descent set,  and a proof that the affine length is preserved along the tower of affine Coxeter groups of type $\tilde A$, which implies in particular that the corresponding tower of affine Hecke algebras is a faithful tower.
	\end{abstract}

\date{\today}

 	\maketitle
	


\section{Introduction} 

{\it This paper is the first of a series in which for the elements  of  an affine Coxeter group  $\tilde W $,  we produce a canonical reduced expression, together with the set of all distinguished representatives of $\tilde W  / W $ where $W$ is a maximal parabolic   subgroup of   $\tilde W $. This very paper is meant to detail  type $\tilde A_{n} $, we do the same in the second paper for types $\tilde C_{n} $ and $\tilde B_{n} $, while type $\tilde D_{n} $ and the five other types are to be treated in the last two.} \\ 

\subsection{ } Coxeter systems and related topics (such as Hecke algebras and their quotients, K-L polynomials and the new born: Light leaves) take a place in the heart of representation theory. Reduced expressions are the salt of such systems: Almost every related object is defined starting from a reduced expression or reduced to a reduced expression explanation, especially and not surprisingly objects which are "independent" from reduced expressions! Such as: Hecke algebras bases and Bruhat order. One may bet that no work concerning/using Coxeter group theory is reduced-expression free. A {\em canonical reduced expression} for elements in the infinite families of finite Coxeter groups has been known while ago, we refer to \cite{St} to see an easy explication of such canonical expressions. $W(\tilde A_{n}) $ is a famous extension of the symmetric group $W(A_n)$, known to be the first "group".

Let $W(A_n)$ the $A$-type Coxeter group with $n\ge 1 $ generators $\{ \sigma_ 1 , \sigma_2,  \dots \sigma_n\}$ (AKA Sym$_{n+1}$). Let $\lfloor i,j \rfloor = \sigma_i \sigma_{i+1} \dots \sigma_j $ for  $1 \le i \le j \le n$. One of the very basic results is: 

\begin{theoremnn}
    $W(A_n)$ is the set of elements of the following canonical reduced form: 
 \begin{equation}\label{1.1}
 \lfloor i_1, j_1 \rfloor \lfloor i_2, j_2 \rfloor \dots  \lfloor i_s, j_s \rfloor  
 \end{equation}
with $n \ge j_1 > \dots > j_s  \ge 1$ and 
$ j_t \ge i_t \ge 1$ for $s \ge t \ge 1$. Identity is to be considered the case where $s=0$.
\end{theoremnn}

This is equivalent to saying that the distinguished representatives of the cosets in $W(A_n)/W(A_{n-1})$
are the  elements $1$ and $\lfloor r, n \rfloor$ for $1\le r \le n$.  \\

In this work we give an analogue of this assertion for the infinite affine Coxeter group  $W(\tilde A_{n}) $. 
More precisely: we give a  canonical reduced expression  for the elements of  this group, with  a full set of the distinguished coset representatives 
of $W(\tilde A_{n}) /W(A_{n}) $. Then we give some examples of direct consequences of this classification by canonical forms. \\ 

\subsection{ }  The key word (and almost everywhere used creature in this work) is {\em affine length} (Definition \ref{AL}): We let $n\ge 2 $ and  $S_n=\{ \sigma_ 1 , \sigma_2,  \dots \sigma_n, a_{n+1}\}$ be the set of Coxeter generators of $W(\tilde A_{n}) $, then the {\it affine length} of an element $w \in W(\tilde A_{n})$ is the minimal number of occurrences of $a_{n+1}$ in all expressions of $w$.   
  We let  $$h(r,i)   =   \sigma_r \sigma_{r+1} \dots \sigma_n \sigma_i \sigma_{i-1} \dots \sigma_1$$ for $ 1\le i \le n-1$, $1 \le  r \le n$, with obvious extension to $r=n+1$ or $i=0$, see \S \ref{affA}.  The set of distinguished representatives of the right $W(A_n)$-cosets  of affine length $1$ is the set of elements  given by the reduced expressions 
$$\mathcal B(r,i)=  h(r, i) a_{n+1}, \quad  0\le i \le n-1,   \   1 \le  r \le n+1$$    (Lemma \ref{Rofw}). We call such expressions 
{\it affine bricks}.
 The main result of this work is Theorem \ref{AA}, of which we give a shortened version as follows: 

\begin{theorem} 
     Any distinguished representative   $w$ of    $W(\tilde A_{n})\backslash W(A_n) $ has a     unique canonical reduced expression:   
\begin{equation}\label{forreferencebelow}
\mathbf{w_a}= \mathcal B(j_1, i_1) \mathcal B(j_2, i_2)   \dots  \mathcal B(j_m, i_m)   \end{equation}
where $m$ is the affine length of $w$ and $(j_s,i_s)_{1\le s \le m}$ is a family of integers satisfying the following {\em pairwise inequalities}:  
\begin{itemize}
\item $ 1\le  j_1 \le n+1$ and  $ 0\le i_1 \le n-1$;     for $ 2\le  s \le m $,   either $i_s=0$ and $j_s=1$, or $ 1\le  i_s \le n-1$ and
$ 1\le  j_s \le n$; 
\item  the sequence $(j_k)$ (resp. $i_k$) is non-increasing (resp.  non-decreasing); 
\item for $ 2\le  s \le m $, if $j_{s-1} > i_{s-1}+1 $, then   $j_s < j_{s-1}   $;  if  
$j_{s} > i_{s}+1 $ then $i_{s} > i_{s-1} $.  
\end{itemize}

Vice versa, any such  family $(j_s,i_s)_{1\le s \le m}$     determines
by (\ref{forreferencebelow})  a distinguished representative 
 $w$    of    $W(\tilde A_{n})\backslash W(A_n) $, in  reduced form, of affine length $m$. We call the very  expression 
$ \mathbf{w_a}:= \mathcal B(j_1, i_1) \mathcal B(j_2, i_2)   \dots  \mathcal B(j_m, i_m) $ {\em    the   affine block}  of  any element in  $wW(A_n)$.  

\end{theorem}

  The proof  establishes in an explicit, algorithmic and independent way the existence of such representatives of minimal length,  given in    canonical form. Appending on the right of an affine block a canonical reduced expression for an element of 
$W(A_n)$ provides a canonical reduced expression for any element in $W(\tilde A_{n})$.   We note that   the lengths of the successive affine bricks in a given affine block form a non-decreasing sequence with first terms  increasing  strictly  up to $n$, and that two of those bricks have the same length if and only if they are identical.    \\

\subsection{ }   We pause here  to  thank the referee of the previous version of this paper who pointed out similarities with section 3.4 in the book \cite{BB} by Bj\"orner and Brenti  on the  one hand, and with the paper \cite{Yilmaz} by  Yilmaz,  \"Ozel, and Ustao\u{g}lu on the other hand. Therefore we studied those references. 

After getting into   the context and language of Gr\"obner-Shirshov bases in   \cite{Yilmaz}, it turns out that the canonical form in Theorem \ref{AA} below is indeed the one given in {\it loc.cit.}  up to taking inverses. Yet, in our work, the single set of parameters is simpler (to read and to use) than the artificially separated parameters  $u$, $v$ and $uv$ in {\it loc.cit.};  the proofs give more insight into  the Coxeter group structure of  
$W(\tilde A_n)$    ({\it loc.cit.} relies on a counting argument); some intermediate calculations are also efficient when working  on consequences. In addition, the present paper is the first in a series in which we will deal with other types of affine Weyl groups, following the same plan as here.

 We turn to the normal form whose existence and uniqueness are established in  \cite[\S 3.4]{BB} for any Coxeter group : it is   the {\it lexicographically first reduced word}, in short the {\it left lex-min form}, for a given order on the set $S$ of generators, hence written 
$S=\{ s_1, \cdots, s_{n+1}\}$ (implicitly and conventionally the lexicographic comparison starts on the left of the word and proceeds from left to right). As observed by Stembridge in \cite[p.1288]{St} (citing Edelman), the normal form 
(\ref{1.1})  for elements of $W(A_n)$   is the reverse, 
  i.e. from right to left,  lexicographically first reduced word, in short the {\it right lex-min form}.   It is  easy to check that  {\em our canonical form is the right lex-min form  for any numbering   $\{ s_1, \cdots, s_{n+1}\}$ of $\{ \sigma_1, \cdots, \sigma_n, a_{n+1}\}$ such that   $s_{n+1}=a_{n+1}$, 
$s_n=\sigma_n$ and $s_{n-1}= \sigma_1$.} 

Our form depends on the choice of   the "affinizing" generator, denoted by $a_{n+1}$; this choice  is essential to define the  affine length and is notably relevant when dealing with the full tower of groups $(W(\tilde A_n))_{n\ge 1}$ as in section \ref{Arr}.  We obtain our canonical form by forcing occurrences of $a_{n+1}$ to be minimal and leftmost. By the previous statement, this implies right-lexicographic minimality (we also  order  the two neighbours of $a_{n+1}$ in the Dynkin diagram -- the effect of this choice is mild, changing it amounts to applying  rules (\ref{productsof2Legobricks})).

Now we make an important remark.  In \cite{BB} existence and uniqueness are a direct consequence of the existence and uniqueness of a minimal element for the lexicographic order. In the present paper, the existence of a form (\ref{forreferencebelow})
  for a distinguished representative of   $W(\tilde A_n)/ W(A_n)$  is easy, but more work has to be done to find necessary and sufficient conditions for such a form to be minimal -- that is, the pairwise inequalities. Getting the general form (\ref{forreferencebelow}),   a product of affine bricks,  from \cite{BB} would be easy, 
but the pairwise inequalities cannot be deduced from  there.

To end this interlude, we thank Bill Casselman for drawing a path for us in the story of normal forms, which developed in the nineties with works of Fokko du Cloux and Bill Casselman in particular   \cite{duCloux_X,Casselman,duCloux_transducer},     and for  pointing out the importance of the result of Brink and Howlett     that Coxeter groups are automatic \cite{BrinkHowlett}. \\ 

\subsection{ } We give three direct consequences of the canonical form. As a  {\bf first consequence},  we show that  through left multiplication by a simple reflection in $S_n$, the canonical form behaves exactly as wished! In other terms: the change made by  left multiplication by a simple reflection is very localized, it happens in at most one affine brick  of the affine  block in such a way that we get a canonical form directly, without passing by the algorithm. This is    Theorem~\ref{lefttimes}, to which we refer  for more detailed statements : 
  
  \begin{theoremnn}[Theorem~\ref{lefttimes}]\label{lefttimesintro}   
Let $ \ \mathbf{w_a}=\mathcal B(j_1, i_1) \mathcal B(j_2, i_2)   \dots  \mathcal B(j_m, i_m)   \  $ be an affine block  of affine length $m\ge 1$, let $w_a$ be the corresponding element of  $W(\tilde A_{n}) $ and let 
$s$ be in $S_n$. Then: 
\begin{enumerate}
\item either $s  w_a$  cannot be expressed by  an affine block, and we have actually 
$l(s w_a)= l(   w_a)+1$  and  $s   w_a=   w_a\sigma_v$ for some $v$, $1\le v \le n$; 

\item  or $s w_a$ has a reduced expression that is an affine block  $ \mathbf{w'_a}$  and,  other than the obvious two cases when  $s=a_{n+1}$ with $ h(j_1, i_1) $ trivial or extremal, the two affine blocks    $  \mathbf{w'_a}$ and $ \mathbf{w_a}$ differ  in one and only one $h(j_s, i_s)$ and one and only one entry there, say   
 $j'_s\ne j_s$ or $i'_s\ne i_s$.   
If $ l(s  w_a)= l( w_a) +1$ we have  $j'_s =j_s-1$ or $i'_s=i_s+1$, while if 
$ l(sw_a)= l( w_a) -1$ we have $j'_s =j_s+1$ or $i'_s=i_s-1$. 
\end{enumerate}
\end{theoremnn}

 This theorem is telling that the canonical form is somehow "stable" by left multiplication by an $s\in S_n$ up to a change in at most one $i_s$ or one $j_s$, but words are but finite sequence of generators! So the canonicity is not bothered by the left multiplications!   
Actually, after getting acquainted with Fokko du Cloux's work as explained above, we saw the similarity of this statement with Theorem 2.6 in \cite{duCloux_transducer}, changing left to right. We chose to leave our statement unchanged with its direct proof, instead of deducing it, however easily, from {\it loc.cit}, because our   proof includes in fact an  automaton to deal with left multiplication of an affine brick, see Lemma  \ref{automaton}.  Even more important, our proof  controls  the path, i.e. the sequence of braid relations,  leading from $ \mathbf{sw_a}$ to  $  \mathbf{w'_a}$, which is essential in an application to light leaves under way.\\

  While for the {\bf  second consequence}:
in section \ref{Rds} devoted to right multiplication,  we compare the descent set $\mathscr{R} (w)$ of $w$ with the descent set $\mathscr{R} (x)$ of $x$, where $w=w_ax$, $x$ in $W(A_n)$,  and $w_a$ is our distinguished representative of $wW(A_{n}) $, having   the affine block $ \mathbf{w_a}$  of $w$ as a reduced expression. We see in Theorem \ref{AA} that $\mathscr{R} (w_a) =  \{ a_{n+1}  \}$. We actually have either  $\mathscr{R} (w)= \mathscr{R} (x)$ or   $\mathscr{R} (w) = \mathscr{R} (x) \cup  \{ a_{n+1}  \}  $.  
We give sufficient conditions on $w$ for $a_{n+1}$ to belong to  $\mathscr{R} (w)$, together with  the {\it hat partner} (see \ref{Bourbaki}) of $a_{n+1}$ multiplied from the right when the multiplication decreases the length.  The cases where $m=1$ and $m=2$ are fully described.\\

A {\bf  third consequence} is to  show that the affine length is preserved in the tower of affine groups defined in \cite{Sadek_Thesis}, that is: When seeing $W(\tilde A_{n-1})$ as a reflection subgroup of $W(\tilde A_{n})$ via the map defined in section   \ref{Arr}:
\begin{eqnarray}
				R_{n}: W(\tilde A_{n-1} ) &\longrightarrow& W(\tilde A_{n} ),\nonumber
							\end{eqnarray}
 then a canonical reduced expression of $(n-1)$-rank is sent to an explicit canonical reduced expression of $(n)$-rank, preserving the affine length, this is Theorem \ref{towerandcanonical}. 

\begin{theoremnn}[Theorem \ref{towerandcanonical}]\label{towerandcanonicalintro}  Let $$
w= h_{n-1}(j_1, i_1) a_{n} h_{n-1}(j_2, i_2) a_{n} \dots  h_{n-1}(j_m, i_m)  a_{n} x, $$ 
 with $x \in W(A_{n-1} )
$,
be the canonical reduced form of an element  $w $  in $ W(\tilde A_{n-1} )$. Then:    
\begin{equation}\label{imageAnminusoneintro}
R_n(w)= h_n(j_1, i_1) a_{n+1} h_n(j_2, i'_2) a_{n+1} \dots  h_n(j_m, i'_m)  a_{n+1}  \lfloor t, n  \rfloor x,  
\end{equation}
where, letting 
$   s=  \max \{k \   /  \    1 \le k \le m  \text{ and }  n-k  - i_k >0 \},$
we have:  
$$
i'_k=i_k  \text{ for } k \le s, \quad i'_k=i_k+1  \text{ for } k > s, \quad   t= n-s+1.  
$$

This implies $L(R_n(w))= L(w)$ and $  l(R_n(w))= l(w)+ 2 L(w),$
hence replacing $a_n$ by $ \sigma_{n} a_{n+1}\sigma_{n} $ in a  reduced expression for $w$ 
  produces a reduced expression for $R_n(w)$ if and only if the expression for $w$ is affine length reduced. 
\end{theoremnn}

The latter theorem gives a necessary and sufficient condition for an element in $ W(\tilde A_{n} )$ to belong to the   image of $ W(\tilde A_{n-1} )$, that is Corollary \ref{con}. A worthwhile consequence is that the corresponding Hecke algebras embed one in the other regardless of the ground ring, that is Corollary \ref{HeckeA}. \\


\subsection{ } We mention briefly farther goals in what follows:

The rigidity of the blocks is a natural field for "cancelling", otherwise called "applying the star operation", to comment this point we need a more advanced calculus, to be done in a forthcoming work centering around the famous Kazhdan-Lusztig cells, and around $W(A_n)$-double cosets since some additional work on the material obtained above (having very strong relations with the second direct consequence) leads to a complete (long) list of canonical reduced expressions of representatives of $W(A_n)$-double classes.

Moreover, in general the canonical form gives us precious data on the space of traces, in particular the embedding of the canonical forms would help a great deal in classifying traces of type Jones on the tower of affine Hecke algebras. Indeed the canonical form given here is easily seen to  coincide (up to a notation), on fully commutative elements, with the normal form (actually, a canonical form) established in 
\cite{Sadek_2016}, which is a crucial ingredient in classifying Markov traces on the tower of affine Temperley-Lieb algebras of type $\tilde A$ in 
 \cite{Sadek_2013_2}. 

In yet another direction, namely an algorithmic way to go towards and come back from the Bernstein presentation, the canonical form indeed gives  long ones easily, definitely the third consequence is a tricky way to shorten the two algorithms. It gives as well a way to enumerate elements by affine length for example.  

Experts of the theory of light leaves (born in \cite{Lib08}) would be interested in such a canonical form, since their computation starts usually with a reduced expression, thus it is even better to have it canonical. For instance, in an ongoing work starting from the canonical form, David Plaza and the author are providing an explicit and simple way to produce  "canonical"  light  leaves bases for the group $W(\tilde A_{n} )$, where usually the construction depends on many non-canonical choices. It is worth to mention that the algorithm to arrive to our canonical form can start from any reduced expression and not only from affine length reduced ones.  \\

The work is self contained and accessible for any who is familiar with Coxeter systems or otherwise want-to-be, we count only on the simplicity of the canonical form, which shows that $ W(\tilde A_{n} )$ is way more "tamed" than Coxeter theory amateurs tend to think, or at least than the author used to think.

\section{A canonical reduced expression}\label{affA}

Let $(W(\Gamma),S)$ be a Coxeter system with associated Coxeter diagram $\Gamma$. Let $w\in W(\Gamma)$ or simply $W$. We denote by $l(w)$ the  length of   $w$ (with respect to $S$).   We define $\mathscr{L} (w) $ to be the set of $s\in S$ such that $l(sw)<l(w)$, in other terms  $s$ appears at the left edge of some reduced expression of $w$.  We define $\mathscr{R}(w)$ similarly, on the right.   If $S'$ is a subset of $S$ and $W'$ is the (parabolic) subgroup of $W$ generated by $S'$, each right coset 
$w W'$ has a unique element of minimal length, say $a$, and for any $x \in W'$ we have $l(ax)=l(a)+l(x)$
(see for instance \cite[Lemma 9.7]{Lusztig}). We call $a$ {\it the distinguished representative} of its coset $a W'$. \\
 
\subsection{Canonical form in 	$W(A_{n})$}	
	Let $n\ge 2$. 		Consider the $A$-type Coxeter group with $n$ generators $W(A_{n})$, with the following Coxeter diagram:
			
			\begin{figure}[ht]
				\centering
				\begin{tikzpicture}[scale=0.6]

  \filldraw (0,0) circle (2pt);
  \node at (0,-0.5) {$\sigma_{1}$}; 
   
  \draw (0,0) -- (1.5, 0);

  \filldraw (1.5,0) circle (2pt);
  \node at (1.5,-0.5) {$\sigma_{2}$};

  \draw (1.5,0) -- (3, 0);

  \node at (3.5,0) {$\dots$};

  \draw (4,0) -- (5.5, 0);
  
  \filldraw (5.5,0) circle (2pt);
  \node at (5.5,-0.5) {$\sigma_{n-1}$};
 
  \draw (5.5,0) -- (7, 0);
  
  \filldraw (7,0) circle (2pt);
  \node at (7,-0.5) {$\sigma_{n}$};

               \end{tikzpicture}
			\end{figure}

			Now let $W(\tilde{A_{n}}) $ be the affine Coxeter group of $\tilde{A}$-type with  set of  $n+1$ generators $ S_n= \left\{ \sigma_{1}, \sigma_{2}, \dots,   \sigma_{n}, a_{n+1}   \right\}$,  perfectly determined   by the  following Coxeter diagram: 
			
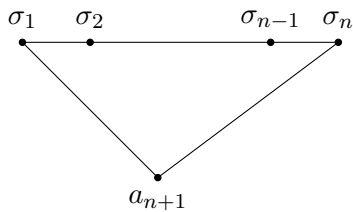
\begin{figure}[h!]
				\centering
				\begin{tikzpicture}[scale=0.6]

 \node at (0,0.5) {$\sigma_{1}$}; 
  \filldraw (0,0) circle (2pt);
   
  \draw (0,0) -- (1.5, 0);
  
  \node at (1.5,0.5) {$\sigma_{2}$};
  \filldraw (1.5,0) circle (2pt);

  \draw (1.5,0) -- (5.5, 0);

  \node at (5.5,0.5) {$\sigma_{n-1}$};
  \filldraw (5.5,0) circle (2pt);
 
  \draw (5.5,0) -- (7, 0);
  
  \node at (7,0.5) {$\sigma_{n}$};
  \filldraw (7,0) circle (2pt);

  \draw (7,0) -- (3, -3);
  
  \filldraw (3, -3) circle (2pt);
  \node at (3, -3.5) {$a_{n+1}$};

  \draw (3, -3) -- (0, 0);
               \end{tikzpicture}
			\end{figure}

  Since $W(A_n)$ is a parabolic subgroup of $W(\tilde A_n)$, we have  for any $v \in W(\tilde A_n)$, $v\ne 1$: 
\begin{equation}\label{parabolic}
\mathscr{R} (v) =  \{ a_{n+1}  \}  \iff  \forall x \in W(A_n) \quad  l(vx)= l(v)+l(x). 
\end{equation}  
			
In the group $W(A_{n})$  we let:  
$$
\begin{aligned}
\lfloor i,j \rfloor &= \sigma_i \sigma_{i+1} \dots \sigma_j   \   \text{ for } n\ge j\ge i \ge 1    \  \text{ and } \    \lfloor n+1,n \rfloor = 1, 
\\
\lceil  i,j \rceil  &= \sigma_i \sigma_{i-1} \dots \sigma_j   \    \text{ for } 1\le j\le i \le n \    \text{ and }  \   \lceil  0,1 \rceil  = 1, 
\\
\qquad  \quad    h(r,i) & =   \lfloor r,n \rfloor \lceil  i,1 \rceil 
  \quad   \text{ for }  0\le i \le n-1, 1\le r \le n+1.  
  \end{aligned}
$$

 \medskip

			One can prove by induction on $n$  (considering right   cosets of $W(A_{n-1})$ in $W(A_{n}) $) the following well-known theorem.

\begin{theorem}\label{1_2}
    $W(A_n)$ is the set of elements of the following canonical reduced form: 
 \begin{equation}\label{Stembridge}
 \lfloor i_1, j_1 \rfloor \lfloor i_2, j_2 \rfloor \dots  \lfloor i_s, j_s \rfloor  
 \end{equation} 
 with $n \ge j_1 > \dots > j_s  \ge 1$ and 
$ j_t \ge i_t \ge 1$ for $s \ge t \ge 1$. Identity is to be considered the case where $s=0$. 
\end{theorem}

 Notice that if $ \sigma_{n} $ appears in form 
(\ref{Stembridge}), then  $ \sigma_{n} $  will certainly appear only once, and it is to be equal to $\sigma_{j_{1}}$.

		\begin{definition}

	An element $u$ in $W(A_{n})$ is called {\rm extremal}  if  both $ \sigma_{n} $ and $ \sigma_{1} $ belong to $Supp(u)$. 
     \end{definition}

\begin{lemma}\label{extremal1}   
Let   $P$ be the parabolic subgroup  of $W(A_{n})$  
generated by   $
 \sigma_2, \dots , \sigma_{n-1} .$  
 An   element in $W(A_{n})$   can uniquely be written in the following reduced form:
$$
 h(r,i) \; x,  \quad 0 \le i \le n-1,  \   1 \le r \le n+1,   \   x \in P. 
$$
The element is extremal if and only if either $r=1$ and $i=0$, or   $i \ge 1$ and $   r \le n  $.
\end{lemma}       			

\begin{proof}  The  elements 
$\lfloor j,n \rfloor$  for  $1 \le j \le n+1$ constitute  the set of distinguished representatives for 
$ W(A_{n})/ W(A_{n-1})$, as is well-known, actually  the first step for proving Theorem \ref{1_2}. 
An easy transformation gives the set of elements   $\lceil  i,1 \rceil $ 
for $0 \le i \le n-1$  as the set of distinguished representatives for $ W(A_{n-1})/P$, hence  the statement.
\end{proof}

  As a consequence, we can define what we call  the {\em extremal canonical form}  of any $w\in W(A_n)$: 
	\begin{equation}\label{extremal}
 h(r,i) \lfloor i_1, j_1 \rfloor \lfloor i_2, j_2 \rfloor \dots  \lfloor i_s, j_s \rfloor  
 \end{equation}
with $1\le r \le n+1 $, $0\le i \le n-1 $, $n-1 \ge j_1 > \dots > j_s  \ge 2$ and 
$ j_t \ge i_t \ge 2$ for $s \ge t \ge 1$.  This form could be used everywhere below 
 instead of  the usual canonical form (\ref{Stembridge}).

\subsection{Affine length}  

\begin{definition}\label{AL}
				We call {\rm	 affine length reduced expression} of a given $u$ in $W(\tilde A_{n})$ any reduced expression with minimal occurrence of $a_{n+1}$, and we {\rm call affine length} of $u$ this minimum, we denote  it by $L(u)$. 
			
			\end{definition}

\begin{remark} We gave the definition of affine length for fully commutative elements in \cite{Sadek_2016}: for such elements the number of occurrences of $a_{n+1}$ in a reduced expression does not depend  on the reduced expression. 

\end{remark} 

\begin{remark}\label{affinelength} The affine length is constant on the double classes of $W(A_{n})$ in $W(\tilde A_{n})$.  It satisfies, for any $v, w \in W(\tilde A_{n})$: 
$$  | L(v)- L(w) | \le L(vw) \le L(v) + L(w).$$ 
\end{remark}

\begin{lemma}\label{lemmafullA}
Let $w$ be in $W(\tilde A_{n})$ with $L(w) =m \ge 2$. 
Fix  an affine length reduced expression of $w$ as follows: 
$$
w =  u_1 a_{n+1} u_2 a_{n+1} \dots u_m a_{n+1} u_{m+1}  
\   \text{ with }  u_i \in W(A_{n})  \text{ for } 1\le i \le m+1 .  
$$ 
Then $u_2, \cdots, u_m$ are extremal and there is a reduced writing of $w$ of the  form: 
 \begin{equation}\label{forme1A}
w = h(j_1, i_1) a_{n+1} h(j_2, i_2) a_{n+1} \dots  h(j_m, i_m)  a_{n+1} v_{m+1},
 \end{equation} 

\noindent
where   $  v_{m+1}$ is  an element in $ W(A_{n})$,  $ 1\le  j_1 \le n+1$, $ 0\le i_1 \le n-1$, and for $ 2\le  s \le m $, either $i_s=0$ and $j_s=1$, or $ 1\le  i_s \le n-1$ and
$ 1\le  j_s \le n$.     
 \end{lemma}    

\begin{proof}
Let $y \in  W(A_{n})$ such that $ a_{n+1} y a_{n+1}$ is  an affine length reduced  expression. 
We use Lemma \ref{extremal1} to write  $y=h(r,i) \; x $  with $ x \in P$. Since $x$ and $a_{n+1}$ 
commute, the element  $ a_{n+1}   h(r,i)   a_{n+1}$ must be affine length reduced. 
Since the braids  
$ a_{n+1} \sigma_1 a_{n+1}$ and $a_{n+1} \sigma_n a_{n+1}$  are to be excluded, both $\sigma_1$ and $\sigma_n$ must appear in $h(r,i)$ so $y$ is extremal.

 Now we proceed from left to right, using Lemma \ref{extremal1}  at each step. We write 
 $u_1= h(j_1, i_1) x_1$ with $x_1 \in P$, so that $u_1 a_{n+1} u_2 = h(j_1, i_1)a_{n+1}  x_1u_2$. We repeat  with $ x_1u_2 a_{n+1}=  h(j_2, i_2)a_{n+1} x_2$   with $x_2 \in P$ and so on, getting (\ref{forme1A}). 	We started with a reduced expression of $w$ so we obtain a reduced expression.
\end{proof}

Yet,  an expression as (\ref{forme1A}) may be reduced without being affine length reduced, as  in the following example: 
$$
a_{n+1} \sigma_n \cdots \sigma_1 a_{n+1} \sigma_1 \cdots \sigma_n a_{n+1} 
=   \sigma_n a_{n+1} \sigma_n \cdots \sigma_1 \cdots \sigma_n a_{n+1} \sigma_n. 
$$

\begin{lemma}\label{Rofw}
An element of affine length $1$ can be written in a unique way as 
$$
h(r,i) a_{n+1} x, \qquad  0 \le  i  \le  n-1, \   1\le r \le n+1, \    x \in W(A_n), 
$$
and such an expression is always reduced. The commutant of $a_{n+1}$ in $W(A_n)$ is $P$. 
 \end{lemma} 

\begin{proof} 	 The existence of such an expression comes from Lemma \ref{extremal1}. 
Showing that the expression is reduced amounts,  by (\ref{parabolic}), to showing that 
$\mathscr{R} (h(r, i) \  a_{n+1} ) =  \{ a_{n+1}  \}  $. Indeed, if   $2\le k \le n-1$, then  $w \sigma_k= h(r, i)\sigma_k a_{n+1}$ has length $l(w)+1$. Now assume  $k=1$ or $k=n$, and $l(w \sigma_k ) < l(w)$. By the exchange condition there is a 
$\sigma_u$   
appearing in $ h(r,i)  $ 
such that $h(r, i) a_{n+1} \sigma_k=  \hat h(r,i) a_{n+1}$ where   $ \hat h(r,i)$ is what becomes $h(r,i)$ after omitting $\sigma_u$. We multiply by $ a_{n+1}$  on the right and get 
$h(r, i) \sigma_ka_{n+1} \sigma_k=  \hat h(r,i) $,  
impossible  considering supports.

Uniqueness amounts to proving that  $h(j,i) a_{n+1} =  h(j',i') a_{n+1} x$ (with obvious notation) implies 
 $x=1$, immediate from   $\mathscr{R} (h(j,i) a_{n+1}) = \{ a_{n+1}\}$ and (\ref{parabolic}).   The last assertion is a consequence of uniqueness. 
\end{proof}

\begin{definition}
We call {\em affine brick} and denote by $ \mathcal B(r, i) $, or  $ \mathcal B_n(r, i) $ when we need to emphasize the dependency in $n$, the expression 
$$
\mathcal B(r, i) =  h(r,i) a_{n+1} , \qquad  0 \le  i  \le  n-1, \   1\le r \le n+1.   
$$
The length of an affine brick $\mathcal B(r, i)$ is $n+1 + i+1 -r$. 
We call an affine brick {\em short} if its length is at most $n$, i.e. $r>i+1$. 
Otherwise we call it {\em long}. 
\end{definition}  

We will keep in mind that the two segments of a {\em short} affine brick commute: 
$$
\mathcal B(r, i) =  \lfloor r,n \rfloor \lceil  i,1 \rceil a_{n+1} = 
 \lceil  i,1 \rceil\lfloor r,n \rfloor a_{n+1}  \quad  \text{ for } r>i+1.   
$$
Other cases are listed in (\ref{productsof2Legobricks}) below.

\subsection{Affine length reduced expressions}  
 
The property  $\mathscr{R} (h(r, i) \  a_{n+1} ) =  \{ a_{n+1}  \}  $  does not extend to elements in form   (\ref{forme1A}) with $v_{m+1}=1$. For instance, 
the relations :
\begin{equation}\label{braidsleftright}
\begin{aligned} 
\sigma_na_{n+1} \sigma_n  \sigma_1a_{n+1} &=      
a_{n+1} \sigma_n   \sigma_1  a_{n+1}   \sigma_1   \\
\sigma_1 a_{n+1} \sigma_n  \sigma_1a_{n+1} &=    
a_{n+1} \sigma_n   \sigma_1  a_{n+1}   \sigma_n 
\end{aligned}
\end{equation}
imply: 
$\sigma_1 \in  \mathscr{R} ( \sigma_na_{n+1} \sigma_n  \sigma_1a_{n+1} )$ and  $\sigma_n \in \mathscr{R} ( \sigma_1a_{n+1} \sigma_n  \sigma_1a_{n+1} )$. 
So  the general form   (\ref{forme1A}) need not be reduced, we must impose more conditions. As in  Lemma \ref{lemmafullA}, we want  to push to the right the simple reflections 
$\sigma_k$, 
$1\le k \le n$, whenever possible.  To do this we bring out the following formulas:

\begin{proposition}\label{exchangeformulas} 
 Let $1\le r \le n+1$, $0 \le u  \le n-1$, $1\le s \le n$ and $1\le v \le n-1$. We have the following rules. \\

\begin{enumerate}
\item 
If  $r > u+1$ and   $s \ge r$: 
$ \quad \mathcal B(r,u)  \mathcal B(s, v)    =  \mathcal B(s+1,u)  \mathcal B(r, v)    \sigma_1. $  \\
\item  If  $s > u+1 \ge v+1$: 
$\quad \mathcal B(r,u)  \mathcal B(s, v)    =  \mathcal B(r,v-1)  \mathcal B(s,u)    \sigma_n. $  \\  
\item   If  $v+1 < s \le u+1$  : 
$\quad  \mathcal B(r,u)  \mathcal B(s, v)    =  \mathcal B(r,v-1)  \mathcal B(s-1,u-1)    \sigma_n. $ \\  
\item  If  $s \le v+1$ and   $v <  u$: 
$\quad  \mathcal B(r,u)  \mathcal B(s, v)    =  \mathcal B(r,v)  \mathcal B(s,u-1)    \sigma_n. $ \\   
\item  If  $r \le u+1<s $: 
$ \quad  \mathcal B(r,u)  \mathcal B(s, v)    =  \mathcal B(s+1,u+1)  \mathcal B(r+1,v)    \sigma_1. $ \\  
\item  If  $r < s \le u+1$: 
$ \quad  \mathcal B(r,u)  \mathcal B(s, v)   =  \mathcal B(s,u)  \mathcal B(r+1,v)    \sigma_1. $ \\  

\end{enumerate}   
 \end{proposition} 

\begin{proof} These are straightforward computations based on (\ref{productsof2Legobricks}),    relying on the rules: 

$  \lfloor r,s \rfloor \ \sigma_k  \ = \  \sigma_{k+1} \   \lfloor r,s \rfloor$ if 
$ r \le k < s$  ; 
$\lceil  r,s \rceil   \ \sigma_k  \ = \  \sigma_{k-1} \    \lceil  r,s \rceil  $ if 
$ r \ge k > s$. 
 
\begin{equation}\label{productsof2Legobricks} 
\begin{aligned}
   \lceil  a,1 \rceil   \lfloor b,n \rfloor &=   \lfloor b-1,n \rfloor    \lceil  a-1,1 \rceil      &\text{ if }    1 < b \le a+1 \le n+1 ;  
\\
   \lceil  a,1 \rceil   \lfloor b,n \rfloor &=   \lfloor b,n \rfloor    \lceil  a,1 \rceil      &\text{ if }   n+1 \ge   b > a+1 ;  
\\
   \lceil  a,1 \rceil   \lfloor 1,n \rfloor &=           \lfloor  a+1,n \rfloor     &\text{ if } 0 \le a \le n ;
\\
   \lfloor a,n \rfloor     \lfloor b,n \rfloor   &=     \lfloor b,n \rfloor   \lfloor a-1,n-1 \rfloor&\text{ if } n+1 \ge  a>b\ge 1  ;  
\\  
 \lfloor a,n \rfloor     \lfloor b,n \rfloor   &=     \lfloor b+1,n \rfloor   \lfloor a,n-1 \rfloor&\text{ if } 1\le  a \le b \le n  ;  
\\
  \lceil  a,1 \rceil   \lceil  b,1 \rceil   &=  \lceil  b,1 \rceil  \lceil  a+1, 2 \rceil  &\text{ if }  1 \le a < b; &
\\
  \lceil  a,1 \rceil   \lceil  b,1 \rceil   &=  \lceil  b-1,1 \rceil  \lceil  a, 2 \rceil    &\text{ if }    a \ge b  .   
\end{aligned}
\end{equation}  
We remark that  equalities (1) to (6)  involve expressions of the same length. They are actually all reduced 
(Lemma \ref{casem2reduced}).  
\end{proof}

\begin{corollary}\label{moreconditions} 
Let $w$ be in $W(\tilde A_{n})$ with $L(w) =m \ge 1$. Among the affine length reduced expressions of $w$: 
$$
w =  u_1 a_{n+1} u_2 a_{n+1} \dots u_m a_{n+1} u_{m+1}  
\   \text{ with }  u_i \in W(A_{n})  \text{ for } 1\le i \le m+1  
$$ 
we fix one with leftmost  occurrences of $a_{n+1}$.  
We have the following, where in (2) and (3) we assume $2 \le s \le m$.  
\begin{enumerate}
\item  For $1\le s \le m$, there exist integers $j_s$, $i_s$ such that $u_s= h(j_s, i_s)$. They satisfy  $ 1\le  j_1 \le n+1$, $ 0\le i_1 \le n-1$, and,  for $ 2\le  s \le m $: 

\centerline{ either $i_s=0$ and $j_s=1$, or $ 1\le  i_s \le n-1$ and
$ 1\le  j_s \le n$.}

\item   If $j_{s-1} > i_{s-1}+1 $, then   $j_s < j_{s-1}   $  and $i_{s} \ge i_{s-1} $; if 
$j_{s} > i_{s}+1 $ then $i_{s} > i_{s-1} $.  

\item    If   
$j_{s-1}\le  i_{s-1}+1 $, then   $j_s \le  j_{s-1}   $  and  $i_{s} \ge  i_{s-1} $ (so $j_{s} \le i_{s}+1 $ also).

\end{enumerate}  
 \end{corollary} 

\begin{proof}  
 All numbered references below refer to Proposition \ref{exchangeformulas}, used to produce  contradictions  to the assumption  that  occurrences of $a_{n+1}$ are leftmost. 

\begin{enumerate}   
\item    follows directly from Lemma \ref{extremal1}  and Lemma \ref{lemmafullA}.

\item   We assume $j_{s-1} > i_{s-1}+1 $. If $j_{s-1}=n+1$ (so $s-1=1$), then $j_s< j_{s-1} $.   If
$j_{s-1} \le n  $ and  $j_s \ge j_{s-1}  $, then (1) gives a contradiction since   the two $a_{n+1}  $ have moved left. Hence $j_s < j_{s-1}   $. 

If also $j_{s} > i_{s}+1 $, then $i_s $ cannot be $ 0$ (since $h(j_s, i_s)$ is extremal), so if $i_{s-1}=0$ we have indeed  $i_s >   i_{s-1}   $.  Now   if   $i_{s-1}>0$ and 
$i_s \le  i_{s-1}   $,  (2) gives a contradiction,  whatever the value of   $j_{s-1}$.

We turn to   $j_{s} \le  i_{s}+1 $.  
If $i_{s-1}=0$ we do have   $i_s \ge i_{s-1}$. If   $i_{s-1}>0$ and   $i_s<  i_{s-1}$, (4) gives a contradiction, 
hence $i_s \ge i_{s-1}$.

\item   We now assume $j_{s-1} \le  i_{s-1}+1 $. If  $j_s > j_{s-1}   $, (5) or (6) give a contradiction.  
 We conclude that $j_s \le j_{s-1}   $.  Now if  $i_{s} <  i_{s-1} $ we are either in case (3) or in case (4), 
and both give a contradiction, so   $i_{s} \ge  i_{s-1} $. 
\end{enumerate} 
\end{proof}

\begin{definition}\label{pairwise}
 Let $m\ge 1$.  A family of integers $(j_s,i_s)_{1\le s \le m}$ is said to satisfy the {\em pairwise inequalities} if the following conditions hold: 

\begin{enumerate}\label{PI}
\item   $ 1\le  j_1 \le n+1$ and  $ 0\le i_1 \le n-1$; 
 
\item   for $ 2\le  s \le m $,   either $i_s=0$ and $j_s=1$, or $ 1\le  i_s \le n-1$ and
$ 1\le  j_s \le n$;

\item   for $ 2\le  s \le m $, we have  $j_s \le  j_{s-1}   $  and  $i_{s} \ge  i_{s-1} $;

\item   If $j_{s-1} > i_{s-1}+1 $, then   $j_s < j_{s-1}   $;

\item   If  
$j_{s} > i_{s}+1 $ then $i_{s} > i_{s-1} $.  
  
\end{enumerate}  

\end{definition} 

We observe that with these conditions $j_{s} > i_{s}+1 $  implies $j_{s-1} > i_{s-1}+1 $.

\begin{theorem}\label{AA} 
   Let $m\ge 1$ and let  $(j_s,i_s)_{1\le s \le m}$  be any family of integers satisfying the  pairwise inequalities. 
The expression 
 $$w= \mathcal B(j_1, i_1) \mathcal B(j_2, i_2)   \dots  \mathcal B(j_m, i_m) 
$$ 
 is reduced and  affine length reduced,  and satisfies
$ \  \mathscr{R} (w) =  \{ a_{n+1}  \} . $     

Any   $w$  in $W(\tilde A_{n})$ with $L(w) =m $ can be written uniquely as 
$$w= \mathcal B(j_1, i_1) \mathcal B(j_2, i_2)   \dots  \mathcal B(j_m, i_m)   x
$$
where $(j_s,i_s)_{1\le s \le m}$ satisfies the pairwise inequalities and   $x$  is 
  the canonical reduced expression of an element in  $W(A_n)$. 
Such a form  is reduced:
 	$$
l(w)= l(x) + \sum_{s=1}^m ( n+1+i_s +1-j_s).
$$
We call the expression $ \mathcal B(j_1, i_1) \mathcal B(j_2, i_2)   \dots  \mathcal B(j_m, i_m) $  
 {\em    the   affine block of   $w$.}    

\medskip

Specifically, a  {\em canonical reduced expression} for $w$ is given by: 
\begin{equation}\label{cancan} w= \mathcal B(j_1, i_1) \mathcal B(j_2, i_2)   \dots  \mathcal B(j_m, i_m)    \lfloor k_1, l_1 \rfloor \lfloor k_2, l_2 \rfloor \dots  \lfloor k_t, l_t \rfloor \end{equation}
with $t \ge 0$,  $n \ge l_1 > \dots > l_t  \ge 1$ and 
$ l_h \ge k_h \ge 1$ for $t \ge h \ge 1$.
\end{theorem}

\begin{proof} The existence of such an expression for $w \in W(\tilde A_{n})$  is given by 
Corollary \ref{moreconditions} and Theorem \ref{1_2}. The other assertions require some work, to be done in the next section.  \end{proof}

\begin{corollary}\label{cos}
The set  $\mathscr{B}_n$ of affine blocks  is the set of canonical reduced expressions for the 
minimal length representatives of the right cosets of $W(A_n)$ in  $W(\tilde A_{n})$. 
\end{corollary} 

  We remark that in an affine block, the affine brick on the left (resp. on the right) of a short affine brick of length $l$ has length at most $l-2$ (resp. at least $l+1$), while the lengths of long affine bricks form a non-decreasing sequence from  left to right, in which  two affine bricks with the same length are equal.

\section{Proof of Theorem \ref{AA}}

\subsection{Skeleton of the proof}\label{3.1}

Let    $j_s$, $i_s$, $1\le s \le m$, be 
 any family of integers satisfying the pairwise inequalities in Definition \ref{pairwise}.  It suffices to prove what we call for short the {\it key statement}:  \\   

{\it    The expression  $  \  w= h(j_1, i_1) a_{n+1} h(j_2, i_2) a_{n+1} \dots  h(j_m, i_m)  a_{n+1} \  $  is reduced and  affine length reduced, and satisfies
$ \  \mathscr{R} (w) =  \{ a_{n+1}  \} $. Furthermore it is the unique such expression of $w$ satisfying the conditions in Theorem \ref{AA}. } \\   

By (\ref{parabolic})  our key statement is equivalent to  the following set of six statements, letting
 $$w_m= h(j_1, i_1) a_{n+1} h(j_2, i_2) a_{n+1} \dots  h(j_m, i_m)  :$$  
\begin{enumerate}
\item The expression $w_m a_{n+1} $ is reduced. 
\item The expression $w_m a_{n+1} \sigma_k $ is reduced for $2\le k \le n-1$.
\item The expression $w_m a_{n+1} \sigma_1 $ is reduced. 
\item The expression $w_m a_{n+1} \sigma_n$ is reduced. 
\item  The element expressed by $w_m a_{n+1} $ has affine length $m$. 
\item The expression $w_m a_{n+1} $ is unique with the given conditions. 
\end{enumerate} 

\medskip  

Our main tool is the criterion  given in Bourbaki   \cite[Ch. IV, \S 1.4]{Bourbaki_1981}. Given a Coxeter system $(W,S)$, we attach to any finite sequence 
$\mathbf s = (s_1, \cdots, s_r)$ of elements in $S$,
 the  sequence $t_\mathbf s= (t_\mathbf s(s_1), \cdots, t_\mathbf s(s_r))$ of elements in $W$ defined by: 
$$ 
t_\mathbf s(s_j)= (s_1\cdots s_{j-1}) \ s_j \   (s_1\cdots s_{j-1})^{-1}   \qquad \text{ for } 1 \le j \le r .   
$$
We call  $t_\mathbf s(s_j)$ the {\it reflection attached to $s_j$} (in the expression $\mathbf s$).
We  shorten the notation sometimes  by writing the expression on the left  into brackets and writing $ [  \dots    ]^{-1}$ for its inverse, namely we write:    
$$ 
t_\mathbf s(s_j)= [s_1\cdots s_{j-1}] \ s_j \   [ \dots  ]^{-1}  .   
$$

 We know  from   \cite[Ch. IV, \S 1, Lemma 2]{Bourbaki_1981}   that  the product  $s_1 \cdots s_r$ is a reduced expression (of the element  $s_1 \cdots s_r$ in $W$)   if and only if   all terms in the sequence  $t_\mathbf s$  are distinct. We will  use this in  the following form: 

\begin{lemma}\label{Bourbaki}
  Let $\mathbf s = (s_1, \cdots, s_r)$ be a sequence of elements in $S$. Assume  that   $s_1 \cdots s_{r-1}$ is a reduced expression. The expression  $s_1 \cdots s_{r}$ is not reduced   if and only if 
there exists $j$,  
$1\le j \le r-1$, such that   $t_\mathbf s(s_j)=t_\mathbf s(s_r)$. Such an integer $j$, if it exists,  is  unique.
\end{lemma}

We remark from the proof in  \cite{Bourbaki_1981} that having  $t_\mathbf s(s_j) =t_\mathbf s(s_r)$ for some $j  \le r-1$ is equivalent to the equality 
$s_1 \cdots s_{j} \cdots s_{r}= s_1 \cdots \hat s_{j}  \cdots \hat s_{r}$ in $W$, where the hat $\hat s_{j} $ over $s_j$ means that $s_j$ is removed from the expression.   We call  for short  the $j$-th element 
$s_j$ of the sequence  the   {\it {\em hat partner}} of $s_r$.

\medskip 
We illustrate the use of this Lemma with the following statement:

\begin{lemma}\label{wplemma}
Let $w \in  W(\tilde A_n) $ and $p \in P$ such that $wp$ is reduced. Then 
$w p a_{n+1}$ is reduced if and only if $w a_{n+1}$ is reduced. 
\end{lemma}  
\begin{proof} The proof by  induction on the length of $p$ is immediate once the length $1$ case is established.  
Assume $w \sigma_k$ is reduced for some $k$, $ 2\le k \le n-1$ and pick a reduced expression $\mathbf w$ for $w$. From Lemma \ref{Bourbaki}, we see that 
$w \sigma_k a_{n+1}$ is not reduced iff there is a simple reflection $s$ in $\mathbf{ w  \sigma_k}$, actually in $\mathbf w$,  such that 
$t_{\mathbf{w \sigma_k a_{n+1}}}(a_{n+1})= t_{\mathbf{w \sigma_k a_{n+1}}}(s)$. Since $\sigma_k$ commutes with $a_{n+1}$ this equality reads exactly $t_{\mathbf{w  a_{n+1}}}(a_{n+1})= t_{\mathbf{w  a_{n+1}}}(s)$ for some $s $ in $\mathbf w$, which is equivalent to $w a_{n+1}$ being not reduced.  
\end{proof}

The proof of  Theorem \ref{AA}, translated into the set of  statements  (1) to (6) above, proceeds by induction on $m$.  The key statement holds 
for $m=1$:  it   is given by 
Lemma \ref{Rofw},  uniqueness follows from Lemma \ref{extremal1}. In subsections \ref{maincomputation} to \ref{ALandunique} we let 
 $m\ge 2$ and, assuming that properties (1) to (6) hold for $w_k$ for any $k \le m-1$, 
 we prove successively properties (1) to  (6) for $w_m$. To do this we  rely on Lemma \ref{Bourbaki}:   
we start with a sequence $\mathbf d = (s_1, \cdots, s_r)$ and a simple reflection $s$ such that 
the expression  $s_1 \cdots s_r$ is reduced and we want to show that   $s_1 \cdots s_r s $ is also 
reduced. We
  transform the reflection  $t_\mathbf d(s) $    attached to $s$ in the expression  $s_1 \cdots s_r s $
 into the reflection attached to some simple reflection  $s'$  in another expression 
$s'_1 \cdots s'_k s' $  which is known to be reduced by induction hypothesis.  

  We recall (\ref{braidsleftright}) and  Corollary \ref{moreconditions}: we need  the pairwise inequalities. In other words: there will be computation,  mostly contained in preliminary lemmas.

\subsection{Rigidity Lemma}
We start with an important Lemma. 
\begin{lemma}[Rigidity Lemma]\label{rigidity}  
Let $w=u \sigma_1 \cdots \sigma_n$ be reduced: $l(w)=l(u)+n$, with $u \in W(\tilde A_n)$. 
Then $ a_{n+1} $ does not belong to $ \mathscr{R} (w)$, in other words  
$u \sigma_1 \cdots \sigma_n a_{n+1} $ is reduced. 
\end{lemma}

\begin{proof}
We proceed by induction on $r=l(u)$, the case $r=0$ being trivial and the case $r=1$ contained in Lemmas \ref{extremal1}. and \ref{Rofw}.  
 We assume that  $r \ge 2$ and that the assertion holds for any $u$ such that $l(u)\le r-1$. We take $u$ with $l(u)=r$ and pick a reduced expression 
$u=s_1 \cdots s_r$, $s_i \in S_n$. 
Assume for  a contradiction that $w a_{n+1}$ is not reduced. By Lemma~3.1 and the induction hypothesis the hat partner of $a_{n+1}$ on the right is the leftmost term $s_1$, i.e. we have the following equality:
\begin{equation}\label{rigidityxxx}
s_1 \cdots s_{r} \sigma_1 \cdots \sigma_n  =  s_2 \cdots s_{r} \sigma_1 \cdots \sigma_n a_{n+1} 
\quad \text{(both sides reduced)} .  
\end{equation}

We discuss according to  $s_r$, which is not equal to $\sigma_1$.

\begin{itemize}

\item[(a)]   If $s_r=\sigma_k$ for $2 < k \le n$,  equality  (\ref{rigidityxxx}) becomes  
$$s_1 \cdots s_{r-1} \sigma_1 \cdots \sigma_n \sigma_{k-1} =  s_2 \cdots s_{r-1} \sigma_1 \cdots \sigma_n a_{n+1} \sigma_{k-1} $$
which, after canceling $\sigma_{k-1}$, contradicts the induction hypothesis for $r-1$. 

\item[(b)]  If $s_r=a_{n+1}$  and $s_{r-1}=\sigma_1$ we transform (\ref{rigidityxxx}) with the braid on 
$a_{n+1}$  and $\sigma_1$ to get 
$$
s_1 \cdots s_{r-2}a_{n+1} \sigma_1 \cdots \sigma_{n-1} a_{n+1} \sigma_n  =  s_2 \cdots  s_{r-2}a_{n+1} \sigma_1 \cdots \sigma_{n-1} a_{n+1} \sigma_n a_{n+1} 
$$
where both sides are reduced, and after using  the braid on 
$a_{n+1}$  and $\sigma_n$ on the right and canceling the rightmost terms we get: 
$$
s_1 \cdots s_{r-2}a_{n+1} \sigma_1 \cdots \sigma_{n-1}    =  s_2 \cdots  s_{r-2}a_{n+1} \sigma_1 \cdots \sigma_{n-1}   \sigma_n  
$$
where both sides are reduced. Using the Dynkin automorphism $a_{n+1} \rightarrow \sigma_1 \rightarrow \sigma_2 \cdots$, denoted by $s \mapsto s'$,  we obtain 
$$
s'_1 \cdots s'_{r-2}  \sigma_1 \cdots \sigma_{n}    =  s'_2 \cdots  s'_{r-2} \sigma_1 \cdots  \sigma_n  a_{n+1} \quad   \text{(both sides reduced)} 
$$
which  contradicts our induction hypothesis for $r-2$. 

\item[(c)]   If $s_r=\sigma_2$   equality  (\ref{rigidityxxx})  reads  
$$s_1 \cdots s_{r-1} \sigma_2  \sigma_1 \cdots \sigma_n =  s_2 \cdots s_{r-1} \sigma_2  \sigma_1 \cdots \sigma_n a_{n+1}  $$
that  transforms under the inverse of the  Dynkin automorphism above, denoted by $s \mapsto s''$, into: 
$$s''_1 \cdots s''_{r-1} \sigma_1  a_{n+1} \sigma_1 \cdots  \sigma_{n-1} =  s''_2 \cdots s''_{r-1} \sigma_1  a_{n+1} \sigma_1 \cdots \sigma_n   $$
where both sides are reduced.  Now  right multiplication  by $a_{n+1}$ clearly reduces the right-hand-side expression, whereas it cannot reduce the left-hand-side expression by (b): observe in the proof of (b) that the induction hypothesis for $r-1$ actually gives us  case (b) for length $r+1$, which is what we need here. 

\item[(d)]  We are reduced to the case where {\it no reduced expression of $u$   falls into cases } (a), (b) or (c). 
Then any reduced expression of $u$ ends (on the right) with $\sigma_n a_{n+1}$ and our claim follows 
from another Lemma: 

\begin{lemma}\label{THErigidelement}
Let $u $ be an element of $W(\tilde A_n)$ of length $r \ge 2$ such that all reduced expressions of $u$ 
end with  $\sigma_n a_{n+1}$ (on the right). Then $u$ is rigid (has a unique reduced expression) and is a left truncation of  
\begin{equation}\label{rigid}
(\sigma_1 \cdots \sigma_n a_{n+1}) ^k    \qquad (k \ge 1) ,  
\end{equation}
which is a rigid hence reduced    expression.  
\end{lemma}

\begin{proof} The rigidity of (\ref{rigid}) is clear and well-known. 
We show by induction that for $2 \le t \le r$, all reduced expressions of $u$ end on the right  with the rightmost $t$ terms in 
$(\sigma_1 \cdots \sigma_n a_{n+1}) ^k $. This holds for $t=2$, we assume it holds up to $t-1$ and prove it for $t\le 3$. Write a reduced expression for $u$ as 
$s_r \cdots s_1$. By induction $s_1$ to $s_{t-1}$ are uniquely determined. There is actually no choice for $s_t$: it cannot commute with $s_{t-1}$ (otherwise we would get another expression with  different rightmost $t-1$ terms), which itself does not commute with  $s_{t-2}$, and  
$s_t s_{t-1}s_{t-2}$ cannot be a braid for the same reason. Hence $s_t$ is the unique neighbour of $s_{t-1}$ in the Dynkin diagram of $\tilde A_n$ different from $s_{t-2}$. 
\end{proof}
 
So in case (d), $w a_{n+1}$ is a left-truncation of some element of form (\ref{rigid}), hence reduced.
\end{itemize}
\end{proof}

\begin{remark} The two lemmas above clearly hold when replacing 
$ \sigma_1 \cdots \sigma_n$ by $ \sigma_n \cdots \sigma_1$, using the Dynkin automorphism of $A_n$.  
\end{remark}

\subsection{A few more lemmas}  We proceed with more lemmas needed in the proof.   

\begin{lemma}\label{almostrigid}
The expression $ D= a_{n+1}  \sigma_{1} \cdots \sigma_n   \cdots \sigma_1 a_{n+1}$  is reduced  and    affine length  reduced.  
\end{lemma}
 \begin{proof} 
The expressions $a_{n+1}  \sigma_{1} \cdots \sigma_n   \cdots \sigma_1 $ and $  \sigma_{1} \cdots \sigma_n   \cdots \sigma_1 a_{n+1}$ are reduced.  If the given expression   was not reduced, the hat partner of the $a_{n+1}$ on the right could only be the $a_{n+1}$ on the left,    contradicting uniqueness in Lemma \ref{Rofw}. 
 Assuming the affine length is $1$, we get 
$a_{n+1}  \sigma_{1} \cdots \sigma_n   \cdots \sigma_1 a_{n+1} = h(r,i) a_{n+1} x$ with the notation in 
Lemma \ref{Rofw}.    Since $x\in P $ is clearly impossible ($a_{n+1} $ cancels out), either $\sigma_1$ or $\sigma_n$ belongs to $ \mathscr{R}(x)$ hence to $ \mathscr{R}(D)$. 
Since $ D$ also equals $ a_{n+1}  \sigma_{n} \cdots \sigma_1   \cdots \sigma_n a_{n+1}$, it is enough to deal with $\sigma_1$, so we assume  $D\sigma_1$ is not reduced. Then  the $\sigma_1$ on the right  has a hat partner $s$ in $D$. This $s$ must be  the $a_{n+1}$ on the left (for otherwise the expression 
$\sigma_{1} \cdots \sigma_n   \cdots \sigma_1 a_{n+1} \sigma_1$  would not be reduced, contrary to 
Lemma \ref{Rofw}).  Transforming the resulting equality with a braid and cancellations we obtain
  $  a_{n+1}  \sigma_{1} \cdots \sigma_n = \sigma_1   \cdots \sigma_n a_{n+1}$, contradicting Lemma \ref{Rofw} again. 
\end{proof}  

\begin{lemma}\label{casem2reduced}
 We consider an expression of the following form: 
$$
h(j_1,i_1) a_{n+1} h(j,i)a_{n+1}, \qquad  0 \le  i_1,i  \le  n-1, \   1\le j_1,j  \le n+1,  
$$
with   $h(j,i)\ne 1$. This expression is reduced except in the  four ``deficient''  cases listed below together with the hat partner of  the rightmost $a_{n+1}$:   
\begin{enumerate}
\item  $h(j,i)=  \lceil  i,1 \rceil $ and $i_1 \ge i \ge 1$, 

the hat partner is the $\sigma_i$ in 
$h(j_1,i_1) =  \lfloor j_1, n  \rfloor \sigma_{i_1} \cdots \sigma_i \cdots \sigma_1$; 
\item $h(j,i)=  \lfloor j, n  \rfloor$ and $1<j \le n $, $j_1\le j$, $i_1 < j-1$, 

the hat partner is the $\sigma_{j}$ in 
$h(j_1,i_1) =   \sigma_{j_1} \cdots \sigma_{j} \cdots \sigma_n  \lceil  i_1,1 \rceil $; 
\item $h(j,i)=  \lfloor j, n  \rfloor$ and $2<j \le n $, $j_1<j$, $i_1 \ge j-1$, 

the hat partner is the $\sigma_{j-1}$ in 
$h(j_1,i_1) =   \sigma_{j_1} \cdots \sigma_{j-1} \cdots \sigma_n  \lceil  i_1,1 \rceil $; 
\item $h(j,i)=  \lfloor 2, n  \rfloor$ and  $j_1=1$, $i_1 =1$, 

 the hat partner is the leftmost  $\sigma_{1}$ in 
$h(j_1,i_1) =   \sigma_{1} \cdots  \sigma_n \sigma_1$.  
\end{enumerate}
In particular, if $h(j,i)$ is extremal, the expression is reduced. 
\end{lemma}  
\begin{proof} 
From Lemma  \ref{Rofw}   we know that $h(j_1,i_1) a_{n+1} h(j,i) $ is reduced. Assume that  $h(j_1,i_1) a_{n+1} h(j,i)  a_{n+1} $ is not. The hat partner of the rightmost $a_{n+1}$ cannot be the leftmost $a_{n+1}$
because the commutant of   $a_{n+1}$ in $W(A_n)$ is $P$.  
So $h(j_1,i_1)$ is not equal to $1$ and the hat partner is a reflection $s$ in $h(j_1,i_1)$.  
Truncating the elements on the left of $s$ we obtain an equality 
$ h(j'_1,i'_1) a_{n+1}  h(j,i)  a_{n+1}= \hat h(j'_1,i'_1) a_{n+1}  h(j,i)   $ where  $ \hat h(j'_1,i'_1)$   is obtained from $ h(j'_1,i'_1)  $ by removing the leftmost reflection. We rewrite this as: 
$$    a_{n+1} h(j'_1,i'_1)^{-1}   \hat h(j'_1,i'_1) a_{n+1} =  h(j,i)  a_{n+1}  h(j,i)^{-1}  . 
$$
Let $V(j'_1,i'_1)$ be the expression on the left hand side. We compute: 
\begin{equation}\label{Vji}
V(j'_1,i'_1)= \left\{ \begin{aligned}  
& \lceil  i'_1,1 \rceil  a_{n+1}  \lfloor 1, i'_1  \rfloor   &\text{ if } j'_1=n+1 ; \cr 
& \lfloor j'_1, n  \rfloor a_{n+1}  \lceil  n,j'_1 \rceil   &\text { if }    1<j'_1 \le n \text{ and } i'_1<j'_1-1 ;   \cr 
&  D   &\text { if }    1<j'_1 \le n \text{ and } i'_1=j'_1-1 ;   \cr 
& \lfloor j'_1+1, n  \rfloor a_{n+1}  \lceil  n,j'_1+1 \rceil   &\text { if }    1<j'_1 \le n \text{ and } i'_1\ge j'_1 ;   \cr 
&  D     &\text { if }    j'_1=1 \text{ and } i'_1\ne 1 ;   \cr 
& \lfloor 2, n  \rfloor a_{n+1}  \lceil  n,2 \rceil  & \text { if }    j'_1=1  \text{ and } i'_1=1.   
\end{aligned}\right. 
\end{equation}
Our equality implies that  $ V(j'_1,i'_1)$ has affine length $1$, which excludes the cases where it is equal to $D$, by Lemma \ref{almostrigid}. The uniqueness in Lemma \ref{Rofw}   now implies that $h(j,i) $ is equal to one of the following:  $\lceil  i'_1,1 \rceil$, $ \lfloor j'_1, n  \rfloor$, $ \lfloor j'_1+1, n  \rfloor$ or $ \lfloor 2, n  \rfloor$, it remains to plug in  the conditions in (\ref{Vji}). 
\end{proof}

\begin{lemma}\label{SubcaseA1} Let $m \ge 2$,  assume the pairwise inequalities hold and    $ j_m >1 $. The element  $h(j_{m-1}, i_{m-1})  \lfloor j_m,n \rfloor $ is reduced and equal to one of the following reduced elements: 
$$
\begin{aligned} 
&  h(j_m, i_{m-1})  \lfloor j_{m-1}-1 ,n-1 \rfloor  &\text{ if }  j_{m-1} > j_{m} > i_{m-1} +1
\\ &  h(j_m-1, i_{m-1}-1)  \lfloor j_{m-1}-1 ,n-1 \rfloor  &\text{ if }  j_{m-1} > i_{m-1} +1 \ge  j_{m} >1  
\\ &  h(j_m-1, i_{m-1})  \lfloor j_{m-1} ,n-1 \rfloor  &\text{ if }  i_{m-1} +1 \ge  j_{m-1} \ge  j_{m} >1  
\end{aligned}
$$ 
Writing this as 
$ 
h(j_{m-1}, i_{m-1})  \lfloor j_m,n \rfloor  =  h(j'_{m-1}, i'_{m-1})  \lfloor u_{m},n-1 \rfloor 
$ 
with $u_{m} \ge 2$,  
 the sequence $\{(j_1,i_1), \cdots, (j_{m-2}, i_{m-2}), (j'_{m-1}, i'_{m-1}) \}$ 
satisfies  the pairwise inequalities. 
\end{lemma}
\begin{proof} We note the following formulas, for  $0 \le a  \le  n-1$, $1 \le b \le n+1$, $1 \le c \le n$: 
\begin{equation}\label{productsof3Legobricks} 
\begin{aligned}
     \lfloor b,n \rfloor    \lceil  a,1 \rceil   \lfloor c,n \rfloor  
&=    \lfloor c,n \rfloor     \lceil  a,1 \rceil    \lfloor b-1,n-1 \rfloor  
 &\text{ if }    c  >  a+1,  b >c   ;
 \\ 
&=      \lfloor b+1,n \rfloor   \lceil  a,1 \rceil  \lfloor b,n-1 \rfloor  
 &\text{ if }    c  >  a+1,  b =c   ; \\ 
&=      \lfloor c-1,n \rfloor   \lceil  a-1,1 \rceil   \lfloor b-1,n-1 \rfloor 
 &\text{ if }   1 <  c  \le  a+1  <b     ;\\
&=      \lfloor c-1,n \rfloor     \lceil  a,1 \rceil    \lfloor b,n-1 \rfloor  
 &\text{ if }   1 <  c  \le  b \le    a+1.    
\end{aligned}
\end{equation} 

They imply the equalities in the Lemma, 
 with $a= i_{m-1} \ge 0$,  $c=j_m>1$,   
$b= j_{m-1}\ge c >1$. The pairwise inequalities are easy to check. The expressions obtained are reduced by Lemma \ref{extremal1} and have the same length that the initial expression. 
\end{proof}

\subsection{The expression $w_m a_{n+1} $ is reduced.}\label{maincomputation}
The case $m=2$ has been dealt with in Lemma \ref{casem2reduced} so we let  $m\ge 3$. 
Furthermore the Rigidity Lemma \ref{rigidity}   gives the result  if $i_m=0$, or if $i_m=n-1$, 
or if $j_m=1$, hence we assume $j_m >1$ and $1 \le i_m < n-1$.

Suppose for a contradiction that  $w_m a_{n+1} $ is not reduced and let $s$ be the hat partner 
of the    $a_{n+1} $ on the right   (Lemma \ref{Bourbaki}). By induction hypothesis the expression 
$h(j_2, i_2) a_{n+1} \dots  h(j_m, i_m)  a_{n+1}$  is reduced so $s$  is to be removed from the leftmost part $h(j_1, i_1) a_{n+1} $.  
From  Lemma \ref{Bourbaki} we have  $  t_{w_m a_{n+1}}(a_{n+1}  ) =    t_{w_m a_{n+1}}(s)$,   
with  $t_{w_m a_{n+1}}(s)=t_{h(j_1, i_1) a_{n+1} }(s)$, so: 
$$
\!  [h(j_1, i_1)  a_{n+1} \dots  h(j_{m}, i_{m}) ]  \  a_{n+1} \    [  \dots    ]^{-1}  =t_{h(j_1, i_1) a_{n+1} }(s). 
$$

Recalling our assumptions $j_m >1$ and $1 \le i_m < n-1$,  we  compute 
$$  \begin{aligned}
 \!     X  &=   [h(j_{m-1}, i_{m-1})a_{n+1}h(j_m, i_m) ] \  a_{n+1}  \  [...]^{-1}     \\  
&=   [h(j_{m-1}, i_{m-1})a_{n+1}  \lfloor j_m,n \rfloor   \lceil  i_m,2 \rceil  ] \  \sigma_1 a_{n+1} \sigma_1 \   [...]^{-1}    \\
&= [h(j_{m-1}, i_{m-1})a_{n+1}  \lfloor j_m,n \rfloor  a_{n+1}   \lceil  i_m,2 \rceil  ] \   \sigma_1 \    [...]^{-1} 
\\
&=      [h(j_{m-1}, i_{m-1}) \lfloor j_m,n \rfloor  a_{n+1} \sigma_n  \lceil  i_m,2 \rceil  ] \   \sigma_1 \    [...]^{-1} 
\end{aligned}
$$

 We let $h(j_{m-1}, i_{m-1}) \lfloor j_m,n \rfloor =h(j'_{m-1}, i'_{m-1}) x$,   $x \in P$,  and   
 $$v= h(j_1, i_1)a_{n+1}   \dots  h(j_{m-2}, i_{m-2})  a_{n+1} h(j'_{m-1}, i'_{m-1}) a_{n+1}      $$
With Lemma \ref{SubcaseA1} we know that the expression $v$  satisfies the conditions in the key statement for $m-1$, so it is reduced and for any  reduced expression $y $ of an element in $W(A_n)$, $v y $ is reduced.   Let $y$ be a reduced form of  $x  \sigma_n  \lceil  i_m,2 \rceil $ ($\sigma_1$ is not in the support). The expression  $v y \sigma_1$ is reduced with leftmost terms $h(j_1, i_1)a_{n+1} $ ($m\ge 3$), so with  Lemma \ref{Bourbaki}
$v y \sigma_1  y^{-1}v^{-1}$ cannot be equal to  $t_{v y \sigma_1}(s)=t_{h(j_1, i_1) a_{n+1} }(s)$, a contradiction with 
$   w_m   a_{n+1}  w_m^{-1} = v y \sigma_1  y^{-1}v^{-1}$.   

 \subsection{The expression $w_m a_{n+1} \sigma_k $ is reduced for $2\le k \le n-1$.}\label{2ton-1}

We just proved that  $w_m a_{n+1}  $ is  reduced, so this follows from Lemmas \ref{extremal1} and \ref{wplemma}.

\subsection{The expression $w_m a_{n+1} \sigma_1 $ is reduced. }\label{casesigma1}

Let $m\ge 2$. We have shown  that    $w_m a_{n+1}  $ is a reduced  expression. 
Suppose for a contradiction that  $w_m a_{n+1} \sigma_1 $  is not and let $s$ be the hat partner of $\sigma_1$  (Lemma \ref{Bourbaki}). By induction hypothesis   $s$ belongs to the leftmost part of the expression: $h(j_1, i_1) a_{n+1} $.   We   have  
$$  t_{   w_m a_{n+1}  \sigma_1   }(\sigma_1  ) =    w_m   a_{n+1}\sigma_1  a_{n+1}   w_m^{-1} =    w_m   \sigma_1  a_{n+1}  \sigma_1 w_m^{-1}=
t_{   w_m   \sigma_1  a_{n+1} }( a_{n+1})$$
while $  t_{   w_m a_{n+1}  \sigma_1   }(s)= t_{   w_m   \sigma_1  a_{n+1} }( s) $ since the two expressions have the same 
 leftmost part $h(j_1, i_1) a_{n+1} $. 

 If  $i_m  = 0$ the expression $w_m   \sigma_1$ is obtained from $w_m$ by replacing $h(j_{m}, 0)$ with  $ h(j_{m}, 1)$. It 
satisfies   the conditions in the key statement, so $ w_m   \sigma_1 a_{n+1}$ is reduced and 
$  t_{   w_m   \sigma_1  a_{n+1} }(a_{n+1}  )  $ cannot be equal to $  t_{   w_m   \sigma_1  a_{n+1} }(s )$.

If $i_m\ge 1$,  we have the following reduced expression for $w_m \sigma_1$: 
 $$\mathbf y= h(j_1, i_1) a_{n+1}  \dots  h(j_{m-1}, i_{m-1})   a_{n+1}   \lfloor j_m,n \rfloor      \lceil i_m,2 \rceil . $$  
A contradiction will follow if we prove that $\mathbf y a_{n+1}$ is reduced or,  equivalently by  Lemma \ref{wplemma},  that 
$$ \mathbf z =  h(j_1, i_1) a_{n+1} \dots  h(j_{m-1}, i_{m-1})   a_{n+1}  \lfloor j_m,n \rfloor   a_{n+1}$$ 
is reduced. Lemma \ref{rigidity}  does the work if $j_m=1$. If $j_m>1$, we observe that 
$$
[ h(j_{m-1}, i_{m-1})a_{n+1}  \lfloor j_m,n \rfloor a_{n+1} ] \   a_{n+1} \   [ \dots ]^{-1} =[  h(j_{m-1}, i_{m-1})  \lfloor j_m,n \rfloor ]  \   \sigma_n  \     [ \dots ]^{-1} . 
$$
  By Lemma \ref{SubcaseA1}, the expression $ h(j_{m-1}, i_{m-1})  \lfloor j_m,n \rfloor$ is reduced hence, by induction, so is $ \mathbf x = h(j_1, i_1) a_{n+1} \dots  h(j_{m-1}, i_{m-1})    \lfloor j_m,n \rfloor$.   
 If $m >2$, we obtain $t_\mathbf x (\sigma_n) = t_\mathbf x (s) $, a contradiction.  If $m=2$, 
we see directly that 
$ \mathbf z =  h(j_1, i_1) a_{n+1}    \lfloor j_2,n \rfloor   a_{n+1} $ is reduced using a braid, Lemma \ref{SubcaseA1} and
Lemma \ref{Rofw}.  

\subsection{The expression $w_m a_{n+1} \sigma_n$ is reduced. }

We follow the same track as for $\sigma_1$ and examine the expression $ w_m   \sigma_n  a_{n+1}$. 

If $i_m=n-1$  and $j_m=n$,   the expression   $w_m   \sigma_n$ is obtained from $w_m$ by replacing 
$h(n, n-1)$ with  $ h(n-1,n-1)$ at the $m$-th rank. It 
satisfies the pairwise inequalities, so $ w_m   \sigma_n a_{n+1}$ is reduced.

If $i_m=n-1$ and $j_m\le n-1$, we have 
$$  \lfloor j_m,n \rfloor      \lceil n-1,1\rceil \sigma_n=  \lfloor j_m,n-2\rfloor      \lceil n,1\rceil =   \lceil n,1\rceil \lfloor j_m+1,n-1 \rfloor  $$ 
and the expression $h(j_1, i_1) a_{n+1} \dots  h(j_{m-1}, i_{m-1})   a_{n+1}     \lceil n,1\rceil \lfloor j_m+1,n-1 \rfloor   a_{n+1}   $ is reduced by Lemmas \ref{wplemma} and \ref{rigidity}.

If $i_m < n-1$,  we have $u:= \lfloor j_m,n \rfloor      \lceil i_m,1\rceil \sigma_n=  \lfloor j_m,n-1 \rfloor      \lceil i_m,1\rceil $. If $j_m>i_m+1$, then $u= \lceil i_m,1\rceil  \lfloor j_m,n-1 \rfloor$; if $j_m\le i_m+1$, then 
$u= \lceil i_m+1,1\rceil  \lfloor j_m+1,n-1 \rfloor$ (\ref{productsof2Legobricks}). The piece $\lfloor \dots ,n-1 \rfloor$ belongs to $P$ and can be left out (Lemma \ref{wplemma}). 
 We get the expression 
 $h(j_1, i_1) a_{n+1} \dots  h(j_{m-1}, i_{m-1})   a_{n+1}    \lceil i',1\rceil  a_{n+1}   $ with 
$i'=i_m$ or $i_m+1$ and $ i_{m-1}< i'$. For $m=2$   
Lemma \ref{casem2reduced} ensures that this expression is reduced (since $i_1< i'$) and we are done. 
For $m> 2$ we let 
$$v=   h(j_{m-1}, i_{m-1})   a_{n+1}    \lceil i',1\rceil  a_{n+1} =  h(j_{m-1}, i_{m-1})       \lceil i',1\rceil  a_{n+1} \sigma_1 .$$ 
If $i_{m-1}=0$, or if  $i_{m-1}\ge 1 $ with (\ref{productsof2Legobricks}), we have : 
$$v=  h(j_{m-1}, i')     \lceil i_{m-1}+1,2\rceil  a_{n+1} \sigma_1. $$   
Since $ h(j_{m-1}, i')$ satisfies  the pairwise inequalities, we get a reduced expression hence the contradiction needed.

\subsection{Affine length and uniqueness}\label{ALandunique}

We already know that an element of  affine length $k$ can be written as 
 $$  \   h(j'_1, i'_1) a_{n+1} h(j'_2, i'_2) a_{n+1} \dots  h(j'_k, i'_k)  a_{n+1} x  \  $$
where $x \in P$ and the  family  of integers  $j'_s$, $i'_s$, $1\le s \le k$,    satisfies the pairwise inequalities, and we just proved that 
for $k\le m$ this expression is reduced. 
Assume for a contradiction that either $w_m a_{n+1} $ has affine length less than $m$, or 
there is another expression of this element satisfying the required conditions. 
Either way,  we have  an integer $k\le m$ and a  family  of integers  $j'_s$, $i'_s$, $1\le s \le k$,    satisfying the pairwise inequalities, such that  
$$
\begin{aligned}
w  &= h(j_1, i_1) a_{n+1} h(j_2, i_2) a_{n+1} \dots  h(j_m, i_m)  a_{n+1}
\\  & =  
h(j'_1, i'_1) a_{n+1} h(j'_2, i'_2) a_{n+1} \dots  h(j'_k, i'_k)  a_{n+1} x 
\end{aligned}
$$
with $x \in P$ and both expressions reduced. We already proved that $ \  \mathscr{R} (w) =  \{ a_{n+1}  \} $,  hence $x=1$ and we can cancel out the term $a_{n+1}$ on the right. 
By induction the element expressed by $w_m=h(j_1, i_1) a_{n+1} h(j_2, i_2) a_{n+1} \dots  h(j_m, i_m) $ 
has affine length $m-1$ and can be uniquely written in this form, so $k=m$ and 
$(j'_s, i'_s)= (j_s, i_s)$ for any $s$, $1 \le s \le m$.

\section{First consequences}
\subsection{Left multiplication }
We need some insight into   left multiplication of affine blocks by a simple reflection.  We recall first a well-known property of  Coxeter groups. 

\begin{proposition}\label{Soergel}
Let $(W,S)$ be a Coxeter group and let $I$ be a strict subset of $S$, generating the parabolic subgroup $W_I$.  Let $W^I$ be the set of minimal coset representatives of $W/W_I$. For $s \in I$ and 
$w \in  W^I$, either $s w \in W^I$, or  there is $r \in I$ such that $sw=wr$ (in particular $l(sw)=l(w)+1$). 
\end{proposition}

\begin{proof}
We only have to prove that  $s w \notin W^I$ implies $sw=wr$ for  some $r \in I$. Let $y \in W^I$ and 
$\alpha \in W_I$ such that $sw=y\alpha$,  $\alpha \ne 1$. It is straightforward to prove that $l(sw)=l(w)+1$, $l(sy)=l(y)+1$ and $l(\alpha)=1$. Let $\mathbf y$ be a reduced expression for 
$y$. Since $s \mathbf y$ and $\mathbf y \alpha$ are reduced expressions but $s\mathbf y \alpha$  is not, the hat partner of $\alpha $ is $s$ and   $w= s y \alpha= y$. 
\end{proof}

In our context, working with  affine blocks, that are canonical reduced expressions for the minimal length representatives of right $W(A_n)$-cosets, we can obtain a more precise statement,   actually a particular case of   \cite[Theorem 2.6]{duCloux_transducer}.

\begin{theorem}\label{lefttimes}  
Let $ \mathbf{w_a}= \mathcal B(j_1, i_1) \mathcal B(j_2, i_2)   \dots  \mathcal B(j_m, i_m)    $ be an affine block  of affine length $m\ge 1$, let $w_a$ be the corresponding element in  $W(\tilde A_n)$, and let 
$s \in S_n$. Then:

\begin{enumerate}

\item   either $s w_a$ is not of  minimal length in its right $W(A_n)$-coset, and we have actually 
$l(s w_a)= l(  w_a)+1$  and  $s w_a= w_a\sigma_v$ for some $v$, $1\le v \le n$;   

\item or $s w_a$ has minimal length in its right $W(A_n)$-coset and one of the following holds:

\begin{enumerate}

\item    $s=a_{n+1}$ and $h(j_1, i_1)=1$, so  $a_{n+1}w_a$  reduces to  the affine block 
$$  \mathcal B(j_2, i_2)   \dots  \mathcal B(j_m, i_m)    \qquad (1 \text{ if } m=1).$$  

\item   $s=a_{n+1}$ and $h(j_1, i_1)$ is extremal,  so $a_{n+1}w_a$ is the affine block 
$$ a_{n+1} \mathcal B(j_1, i_1) \mathcal B(j_2, i_2)   \dots  \mathcal B(j_m, i_m) .$$  

\item Otherwise, $s w_a$ is expressed as  an affine block of the following   form: 
$$ \mathcal B(j'_1, i'_1) \mathcal B(j'_2, i'_2)   \dots  \mathcal B(j'_m, i'_m) $$ 
where the $2m$-tuples  $(j_1, i_1, \cdots, j_m, i_m)$ and $(j'_1, i'_1, \cdots, j'_m, i'_m)$  
differ in one and only one entry, say   
 $j'_r\ne j_r$ or $i'_r\ne i_r$.   
If $ l(sw_a)= l(w_a) +1$ we have  $j'_r =j_r-1$ or $i'_r=i_r+1$, while if 
$ l(s w_a)= l(w_a) -1$ we have $j'_r =j_r+1$ or $i'_r=i_r-1$. 
\end{enumerate} 
\end{enumerate} 
\end{theorem}

\begin{remark}
In the case when $ l(s w_a)= l(w_a) -1$,  Theorem \ref{lefttimes} says that the ``hat partner'' of $s$ is a $\sigma_{j_r}$ or a $\sigma_{i_r}$ and that the resulting expression is in canonical form, i.e.  an affine block.  
\end{remark}

\begin{proof}  We establish first our statement in the case when $s=\sigma_u$ with  $1\le u \le n$. 
The case of affine length 1 is detailed in the following Lemma,   easily checked, in fact an automaton describing left multiplication of an affine brick $\mathcal B(j,i)$ by  
$\sigma_u $.  The result is either $\mathcal B(j,i)\sigma_v$ for some $v$, or an affine  brick of length $l(\mathcal B(j,i) \pm 1$.

\begin{lemma}\label{automaton} 
 For $1 \le j \le n+1$ and $n-1 \ge i \ge 0$,    we have if $1\le u \le n$:\\
$$
\sigma_u \mathcal B(j,i)= \left\{ \begin{matrix} \mkern-36mu \mathcal B(j,i)  \sigma_{u-1} \quad  \text{ if } 
\left\{ \begin{matrix} \mathcal B(j,i)   \text{ short and } j<u  ,  \quad \  \cr \mathcal B(j,i)   \text{ long and } i+2 < u ; 
 \end{matrix} \right.  \qquad  \quad \qquad
\cr   \cr 
\mkern-36mu \mathcal B(j,i)  \sigma_u \quad  \text{ if } 
\left\{ \begin{matrix} \mathcal B(j,i)   \text{ short and } i+1<u <j-1  ,  \cr \mathcal B(j,i)   \text{ long and } j < u <i+1 ;  \qquad  \  
 \end{matrix} \right.  \qquad 
\cr  \cr 
\mkern-36mu \mathcal B(j,i)  \sigma_{u+1} \quad  \text{ if } 
\left\{ \begin{matrix} \mathcal B(j,i)   \text{ short and } u <i  , \ \quad  \  \cr \mathcal B(j,i)   \text{ long and }  u < j-1 ; 
 \end{matrix} \right.  \qquad  \quad \qquad  
\cr \cr  
 \mkern-36mu \mathcal B(j-1,i) \quad   \text{ if } 
 u= j-1 ,  \qquad \qquad \qquad \qquad \qquad \quad   \qquad 
\cr   \cr  \mkern-36mu \  \qquad 
 \mathcal B(j,i+1)   \  \     \text{ if } 
\left\{ \begin{matrix} \mathcal B(j,i)   \text{ short and }i+1 <j-1  \text{ and }  u=i+1 ,  \cr \mathcal B(j,i)   \text{ long and }   u =i+2 ; \qquad \qquad  \qquad  \qquad 
 \end{matrix} \right.    \cr \cr  
\mkern-36mu  \mathcal B(j,i-1) \quad   \text{ if } 
\left\{ \begin{matrix} \mathcal B(j,i)   \text{ short and }   u=i  , \quad  \cr \mathcal B(j,i)   \text{ long and }   u =i+1 ;  
 \end{matrix} \right.  \qquad \quad \qquad 
\cr   \cr 
\mkern-36mu \mathcal B(j+1,i)   \quad   \text{ if } u=j.  \qquad  \qquad \qquad \qquad \qquad \qquad \    \qquad 
 \end{matrix}  \right.
$$

In particular  $\mathscr{L} (\mathcal B(j,i) )$ is the set 
$   \{   \sigma_j, \sigma_{i}\}  $  if $\mathcal B(j,i) $ is short, the set $   \{   \sigma_j, \sigma_{i+1}\}  $ if $\mathcal B(j,i)$ is long.
   \end{lemma}

We prove the general case by induction on $m$. Assuming the assumptions hold up to $m-1\ge 1$, we let 
$w'_a= \mathcal B(j_1, i_1) \mathcal B(j_2, i_2)   \dots  \mathcal B (j_{m-1}, i_{m-1})   $ and study  
$\sigma_u w_a =  (\sigma_u w'_a )  \mathcal B(j_m, i_m)   $ according to the shape of $\sigma_u w'_a $. 

\begin{itemize}
\item 
If $\sigma_u w'_a $ is not of minimal length in its coset, we write $\sigma_u w'_a= w'_a \sigma_v $ for some $v$, $1\le v \le m$, so that 
$$\sigma_u w_a =   w'_a  \sigma_v  \mathcal B(j_m, i_m)   .$$  
We deal with $  \sigma_v   \mathcal B(j_m, i_m)   $ using   the previous Lemma. If some $\sigma_z$ appears on the right we are in case (1). Assume now $  \sigma_v  \mathcal B(j_m, i_m)   =  \mathcal B(j'_m, i'_m) $. If  
  $j'_m=j_m-1$ or $i'_m=i_m+1$, we are in case (2c) since we get an affine block.  
 If  
  $j'_m=j_m+1$ or $i'_m=i_m-1$, it seems at first that the resulting expression  might not be canonical,  depending on the value of $j_{m-1}$ or $i_{m-1}$. But actually the expression has no other choice than being canonical. Indeed we are in a case where $ l(\sigma_u w_a)= l(w_a) -1$, hence 
$\sigma_u w_a$ has minimal length in its right coset and by Proposition     \ref{exchangeformulas}    the required inequalities are  satisfied.   

\item
If $\sigma_u w'_a $ is of minimal length in its coset, we write it as an affine block and get 
$$\sigma_u w_a =    \mathcal B(j'_1, i'_1) \mathcal B(j'_2, i'_2)   \dots  \mathcal B (j'_{m-1}, i'_{m-1})      \mathcal B(j_m, i_m)   .$$ 
This is an affine block except possibly when the only difference between the $i,j$'s and the $i', j'$'s happens for $j'_{m-1}$ or $ i'_{m-1}$ and the resulting pairs $ (j'_{m-1}, i'_{m-1}) $ and $(j_m, i_m)$ 
do not satisfy the required inequalities. 
In such a case we apply Proposition 
\ref{exchangeformulas} and get 
$$\sigma_u w_a =   \mathcal B(j_1, i_1) \mathcal B(j_2, i_2)   \dots  \mathcal B(j''_{m-1}, i''_{m-1}) \mathcal B(j''_m, i''_m)   \sigma_t$$ 
with $t=1$ or $n$. Proposition \ref{Soergel} leaves only one choice, namely 
$\sigma_u w_a  = w_a \sigma_t$. This finishes the proof in the case $s=\sigma_u$. 
\end{itemize}

We take next $s=a_{n+1}$.   The cases when $h(j_1, i_1) $ is extremal or equal to $1$ are obvious. Otherwise we have $h(j_1, i_1) =   \lfloor j_1, n  \rfloor$ with $1< j_1\le n$ or 
$h(j_1, i_1) =   \lceil i_1,1 \rceil  $ with $i_1\ge 1$. Using a braid we reduce the claim to 
the one we have already proved for $s=\sigma_n$ or $s=\sigma_1$, left-multiplying  the affine block starting at $h(j_2, i_2)$. Checking that the resulting expression satisfies the pairwise inequalities  is 
straightforward and left to the reader. 
\end{proof}

\subsection{Right descent set}\label{Rds}
In this subsection we study the right descent set  $\mathscr{R} (w) $ of an element 
  $w$  in $W(\tilde A_{n})$ with $L(w) =m > 0 $,  given canonically as 
$$
w= \mathcal B(j_1, i_1) \mathcal B(j_2, i_2)   \dots  \mathcal B(j_m, i_m)   x, \qquad x \in W(A_n), 
$$
(hence the family $(j_s,i_s)_{1\le s \le m}$ satisfies the pairwise inequalities).  
 
The first observation is the following:  $  \    \mathscr{R} (x) \subseteq \mathscr{R} (w) \subseteq  \mathscr{R} (x) \cup  \{ a_{n+1}  \}  .$   
 Indeed  if a simple reflection $s$ other than $a_{n+1}$ does not belong to $ \mathscr{R} (x)$, then $ws$ is reduced by Theorem \ref{AA}. 

The determination of $\mathscr{R} (w) $ then amounts to 
giving the conditions for $ a_{n+1} $ to belong to this set. 
 Writing $x=h(j,i)p$, $p \in P$,   Lemma \ref{wplemma} shows that these conditions   depend only on the  $h(j,i)$ part of $x$, not on $p$.  Of course 
Theorem \ref{AA}  ensures that if $(j_m, i_m), (j,i)$ satisfy the pairwise  inequalities, then $ a_{n+1} $does not belong to $\mathscr{R} (w) $. 
  It is tempting to believe that if $x$ is extremal, then 
$w a_{n+1} $ is reduced.   This holds for $m=1$ (Lemma \ref{casem2reduced}) but it is not true in general, as we can see in the following Lemma that gives a full account of the case $m=2$.

\begin{lemma}\label{length2andx}
We consider an expression of the following form: 
$$
\mathcal B(j_1, i_1) \mathcal B(j_2, i_2)  x a_{n+1}  $$

\noindent 
where   $   x \in W(A_n)$ and $(j_1, i_1), (j_2,i_2)$ satisfy the pairwise inequalities,   
and we write $x=h(j,i)p$, $p \in P$.   If  $h(j,i)\ne 1$ this expression is reduced except:  

\begin{itemize}

\item  in the  four ``deficient''  cases listed in  Lemma \ref{casem2reduced}, with $j_1, i_1$ replaced by $j_2, i_2$, 

\item in  the  cases 
 listed below together with the hat partner of  the rightmost $a_{n+1}$:   
\begin{enumerate}
\item  $h(j,i)= \sigma_n \sigma_1 $ and $j_2>1$ and  $1\le i_2 < n-1$, 

the hat partner is the leftmost  $a_{n+1}$; 

\item $h(j,i)= h(n,i)$ and $1\le i \le i_2 < n-1$, $i < j_2$, \text{ and } $i_1 \ge i-1$,               

the hat partner is the $\sigma_{i-1}$ in 
$h(j_1,i_1) =    \lfloor j_1, n  \rfloor  \sigma_{i_1} \cdots \sigma_{i-1} \cdots \sigma_1  $; 

\item $h(j,i)= h(n,i)$ and $1 \le i \le i_2 < n-1$, $i \ge j_2$, \text{ and } $i_1 \ge i$,               

the hat partner is the $\sigma_{i}$ in 
$h(j_1,i_1) =    \lfloor j_1, n  \rfloor  \sigma_{i_1} \cdots \sigma_{i} \cdots \sigma_1  $.    
\end{enumerate}
\end{itemize}
We note that in  cases (1), (2), (3)  above, the element $x$ is extremal. 
\end{lemma}  

We skip the (technical) proof of this Lemma. 
 Further computation shows that for $m=3$ the list of non reduced cases grows bigger, therefore we do not pursue this matter for now. 

Observing that actually, for $m\ge 2$:  
 $$    \mathscr{R} (x) \! \subseteq  \mathscr{R} (\mathcal B(j_{\scriptscriptstyle m}, i_{\scriptscriptstyle m})      x) 
\! \subseteq  \mathscr{R} (\mathcal B(j_{\scriptscriptstyle m-1}, i_{\scriptscriptstyle m-1}) \mathcal B(j_{\scriptscriptstyle m}, i_{\scriptscriptstyle m})   x)  \!  \subseteq \mathscr{R} (w)    
\!\subseteq  \mathscr{R} (x) \cup  \{ a_{\scriptscriptstyle  n+1}  \}  $$
we draw  from  Lemmas \ref{casem2reduced} and \ref{length2andx} a 
{\em list of cases in which $a_{n+1}$ does belong to $ \mathscr{R} (w)$,} together with its hat partner: 
\begin{enumerate}
\item 
\begin{enumerate}
\item  $h(j,i)=  \lceil  i,1 \rceil $ and $i_m \ge i \ge 1$, 

the hat partner is the $\sigma_i$ in 
$h(j_m,i_m) =  \lfloor j_m, n  \rfloor \sigma_{i_m} \cdots \sigma_i \cdots \sigma_1$; 
\item $h(j,i)=  \lfloor j, n  \rfloor$ and $1<j \le n $, $j_m\le j$, $i_m < j-1$, 

the hat partner is the $\sigma_{j}$ in 
$h(j_m,i_m) =   \sigma_{j_m} \cdots \sigma_{j} \cdots \sigma_n  \lceil  i_m,1 \rceil $; 
\item $h(j,i)=  \lfloor j, n  \rfloor$ and $2<j \le n $, $j_m<j$, $i_m \ge j-1$, 

the hat partner is the $\sigma_{j-1}$ in 
$h(j_m,i_m) =   \sigma_{j_m} \cdots \sigma_{j-1} \cdots \sigma_n  \lceil  i_m,1 \rceil $; 
\item $h(j,i)=  \lfloor 2, n  \rfloor$ and  $j_m=1$, $i_m =1$, 

 the hat partner is the leftmost  $\sigma_{1}$ in 
$h(j_m,i_m) =   \sigma_{1} \cdots  \sigma_n \sigma_1$.  
\end{enumerate}
 
\item 
\begin{enumerate}
\item  $h(j,i)= \sigma_n \sigma_1 $ and $j_m>1$ and  $1\le i_m < n-1$, 

the hat partner is the $a_{n+1}$ on the  left of $ h(j_m, i_m)$; 

\item $h(j,i)= h(n,i)$ and $1\le i \le i_m < n-1$, $i < j_m$, \text{ and } $i_{m-1} \ge i-1$,               

the hat partner is the $\sigma_{i-1}$ in 

$h(j_{m-1},i_{m-1}) =    \lfloor j_{m-1}, n  \rfloor  \sigma_{i_{m-1}} \cdots \sigma_{i-1} \cdots \sigma_1  $; 

\item $h(j,i)= h(n,i)$ and $1 \le i \le i_m < n-1$, $i \ge j_m$, \text{ and } $i_{m-1} \ge i$,               

the hat partner is the $\sigma_{i}$ in 

$h(j_{m-1},i_{m-1}) =    \lfloor j_{m-1}, n  \rfloor  \sigma_{i_{m-1}} \cdots \sigma_{i} \cdots \sigma_1  $.    
\end{enumerate}
\end{enumerate}

We point out again that this list is not exhaustive if $m\ge 3$.

\subsection{A tower of canonical reduced expressions }\label{Arr}

We study the affine length in the tower of injections $W(\tilde A_{n-1} ) \hookrightarrow W(\tilde A_{n})$
built with  the group monomorphism 
								\begin{eqnarray} 
				R_{n}: W(\tilde A_{n-1} ) &\longrightarrow& W(\tilde A_{n} )\nonumber\\
				\sigma_{i} &\longmapsto& \sigma_{i} \text{ for } 1\leq i\leq n-1\nonumber\\
				a_{n} &\longmapsto& \sigma_{n} a_{n+1}\sigma_{n}  \nonumber
			\end{eqnarray}
  from    \cite[Lemma 4.1]{Sadek_2016}. We produce below the canonical reduced expression of
$R_n(w)$ given the canonical reduced expression of $w \in   W(\tilde A_{n-1} )$ from Theorem \ref{AA}. In particular, 
$R_n(w)$ and $w$ have the same affine length and 
 the Coxeter length of $R_n(w)$ is fully determined by the Coxeter length and affine length of $w$. 

In this subsection we need to include the dependency on $n$ in the notation, so we write 
$ h_n(r,i)  =   \lfloor r,n \rfloor \lceil  i,1 \rceil $. 

\begin{theorem}\label{towerandcanonical}  Let  
$$
w= h_{n-1}(j_1, i_1) a_{n} h_{n-1}(j_2, i_2) a_{n} \dots  h_{n-1}(j_m, i_m)  a_{n} x 
$$
be the canonical reduced expression of an element  $w $  in $ W(\tilde A_{n-1} )$, where $x$ is the canonical reduced expression of an element in $W(A_{n-1})$. Substituting $ \sigma_{n} a_{n+1}\sigma_{n} $ for $a_n$ in this expression produces a reduced expression which can be transformed into the   canonical reduced expression of $R_n(w)$, that  has the following shape:

\begin{equation}\label{imageAnminusone}
R_n(w)= h_n(j_1, i_1) a_{n+1} h_n(j_2, i'_2) a_{n+1} \dots  h_n(j_m, i'_m)  a_{n+1}  \lfloor t, n  \rfloor x  
\end{equation}
where, letting 
$$   s=  \max \{k \   /  \    1 \le k \le m  \text{ and }  n-k  - i_k >0 \} ,$$ 
we have:  
$$
i'_k=i_k  \text{ for } k \le s, \quad i'_k=i_k+1  \text{ for } k > s, \quad   t= n-s+1.  
$$
This implies $$L(R_n(w))= L(w), \qquad  l(R_n(w))= l(w)+ 2 L(w),$$
hence replacing $a_n$ by $ \sigma_{n} a_{n+1}\sigma_{n} $ in a  reduced expression for $w$ 
  produces a reduced expression for $R_n(w)$ if and only if the expression for $w$ is affine length reduced. 
\end{theorem}

Note that  we have  $s \le n-1$.  

\begin{proof} We observe first that the expression (\ref{imageAnminusone}) given for $R_n(w)$ is  canonical:  the pairwise inequalities are clearly satisfied, and  the fact that 
$ \lfloor t, n  \rfloor x$, $  x \in W(A_{n-1})$, is reduced, has been used since the beginning of this paper.  The last part of the Proposition states immediate consequences. We only have to produce form 
(\ref{imageAnminusone}).

Substituting $ \sigma_{n} a_{n+1}\sigma_{n} $ for $a_n$ in the canonical reduced expression of $w$  gives: 
$$
R_{\scriptscriptstyle n} (w)= h_{\scriptscriptstyle n-1} (j_{\scriptscriptstyle 1} , i_{\scriptscriptstyle 1} )  \sigma_{n} a_{\scriptscriptstyle n+1} \sigma_{n} h_{\scriptscriptstyle n-1} (j_{\scriptscriptstyle 2} , i_{\scriptscriptstyle 2} )  \sigma_{n} a_{\scriptscriptstyle n-1} \sigma_{n} \dots  h_{\scriptscriptstyle n-1} (j_{\scriptscriptstyle m} , i_{\scriptscriptstyle m} )   \sigma_{n} a_{\scriptscriptstyle n-1} \sigma_{n} x.  
$$
For the leftmost term, we have $ h_{n-1}(j_1, i_1)  \sigma_{n} =  h_{n}(j_1, i_1) $ since 
$i_1 \le n-2$. For the next one we have 
$$\sigma_{n} h_{n-1}(j_2, i_2)  \sigma_{n} =  \lfloor j_2, n-2  \rfloor \sigma_{n}\sigma_{n-1}\sigma_{n}  \lceil i_2,1 \rceil  =  \lfloor j_2, n  \rfloor \sigma_{n-1}   \lceil i_2,1 \rceil . $$
If $i_2=n-2$, we obtain $h_{n}(j_2, n-1) $, otherwise $ \sigma_{n-1}  $ travels to the right; so if $m=1$ or $m=2$ our claim holds. Assuming the claim holds up to $m-1  \ge 2$, we prove it for $m$. 
Let $  s= s_{m-1}=  \max \{k \   /  \    1 \le k \le m-1  \text{ and }  n-k  - i_k >0 \} $ and $ t_{m-1} = n-s_{m-1}+1$. We 
  have 
$$
R_{\scriptscriptstyle n} (w)= h_{\scriptscriptstyle n} (j_{\scriptscriptstyle 1} , i_{\scriptscriptstyle 1} ) a_{\scriptscriptstyle n+1}  \dots  h_{\scriptscriptstyle n} (j_{\scriptscriptstyle m-1} , i'_{\scriptscriptstyle m-1} )  a_{\scriptscriptstyle n+1}   \lfloor t_{\scriptscriptstyle m-1} , n  \rfloor  h_{\scriptscriptstyle n-1} (j_{\scriptscriptstyle m} , i_{\scriptscriptstyle m} )   \sigma_{\scriptscriptstyle n}  a_{\scriptscriptstyle n+1} \sigma_{\scriptscriptstyle n}  x.  
$$

We show first: $ t_{m-1} > j_m$. Indeed we have    $ t_{m-1} > i_s+1$ -- in particular 
$ t_{m-1} -1 > 1$, to be used soon. If $j_s \le i_s+1$ we are done, otherwise the sequence 
$(j_r)$ decreases strictly for $r \le s+1$ hence $j_{s+1} \le n-(s+1) +1 <  t_{m-1}$.

We can now  compute: 
$$
  \lfloor t_{m-1}, n  \rfloor  h_{n-1}(j_m, i_m)   \sigma_{n} 
=    \lfloor j_m, n  \rfloor   \lfloor t_{m-1}-1, n  -1  \rfloor    \lceil i_m,1 \rceil 
$$
equal to 
\begin{enumerate}
\item  $\lfloor j_m, n  \rfloor      \lceil i_m,1 \rceil \lfloor t_{m-1}-1, n  -1  \rfloor $ 
if $t_{m-1}-1 > i_m+1$  ;  
\item    $\lfloor j_m, n  \rfloor    \lceil i_m +1,1 \rceil     \lfloor t_{m-1}, n  -1  \rfloor $
if $t_{m-1}-1 \le  i_m+1$. 
\end{enumerate}
Recalling $ t_{m-1} -1 > 1$, we obtain in these two cases, respectively: 
\begin{enumerate}
\item $
R_n(w)=$ 

$ h_n(j_1, i_1) a_{n+1} \dots  h_n(j_{m-1}, i'_{m-1})  a_{n+1}    h_{n}(j_m, i_m)    a_{n+1}\lfloor t_{m-1}-1, n    \rfloor  x 
$;  
\item   $
R_n(w)=$ 

$ h_n(j_1, i_1) a_{n+1} \dots  h_n(j_{m-1}, i'_{m-1})  a_{n+1}    h_{n}(j_m, i_m+1)    a_{n+1}\lfloor t_{m-1}, n    \rfloor  x 
$. 
\end{enumerate}
Both have the expected form, by induction, once we observe the following. If $  i'_{m-1}  = i_{m-1}+1$, then 
also $i'_m= i_m+1$:  certainly $  i'_{m-1}  = i_{m-1}+1$ implies $t_{m-1}= t_{m-2} \le i_{m-1}+2$.  Hence 
$t_{m-1} \le i_{m}+2$,  so finally $t_{m-1}=t_m$ and $i'_m= i_m+1$. 
 \end{proof}
 
 \begin{corollary}\label{con}
 
 Let $w  \in W(\tilde A_n)$ be given in its canonical form: 
 $$w= h(j_1, i_1) a_{n+1} h(j_2, i_2) a_{n+1} \dots  h(j_m, i_m)  a_{n+1}  x, \quad x \in W(A_n),  $$    
 then $w \in R_n(W(\tilde A_{n-1} ))$ if and only if  the following conditions hold:
 
\begin{enumerate}
 
  \item $j_1 \le  n$  and $i_1 < n-1$; 
   
   \item letting $s=  \max \{k \   /  \    1 \le k \le m  \text{ and }  n-k  - i_k >0 \}$, we have:

$ n-(s+1)  - i_{s+1} < 0$;

   \item $x=\lfloor n-s+1, n \rfloor .y$ with   $y \in W(A_{n-1})$.
 \end{enumerate}

 \end{corollary}

\begin{proof} The only thing to check is that, letting   $\bar i_t=i_t$ if $t\le s$ and $\bar i_t=i_t-1$ if $t>s$, the family  $(j_t, \bar i_t)_{1\le t \le m}$ satisfies the pairwise inequalities. This is left to the reader. 
\end{proof}
 
 The corollary tells that for a $w$ in $ W(\tilde A_n)$:  belonging to the image $R_n(W(\tilde A_{n-1} ))$ depends only on the $n$ leftmost affine bricks of the affine block $\mathbf{w_a}$  of $w$ and the finite part $x\in W(A_n)$! And that for every affine block $\mathbf{w_a}$ verifying  conditions (1) and (2) there are exactly $n!$ elements  $x\in  W(A_n)$ such that $\mathbf{w_a}.x$ is in  $R_n(W(\tilde A_{n-1} ))$. And finally that every element in $W(\tilde A_{n-1} )$ can be attained in such a way. \\

We can deduce from this the faithfulness of the tower of Hecke algebras on any ring, following the tracks of   \cite[Theorem 3.2]{Sadek_2019}, with exactly the same proofs. In what follows, by algebra we mean
  $K$-algebra, where   $K$ is an arbitrary commutative ring with identity. We  fix  an invertible element $q$ in $K$.  
  There is a unique algebra structure on the free $K$-module with basis $ \{  g_w  |  w \in W(\tilde A_n) \} $ satisfying for $s \in S_n$: 
\begin{equation*}\label{definingrelations} 	
	 \begin{aligned}
		 &g_{s} g_{w} =g_{sw}     ~~~~~~~~~~~~~~~~~~~~  \text{ if } s \notin \mathscr{L} (w) , \\
		 &g_{s} g_{w} =qg_{sw}+ (q-1)g_w ~~  \text{ if } s \in \mathscr{L} (w). 
			  \end{aligned}    \qquad 
		\end{equation*}
This algebra is the Hecke algebra of type $\tilde A_n$, denoted   by    $H \tilde A_n(q)$. It  has a presentation  given by generators $ \left\{  g_s \ | \   s \in S_n \right\} $ and well-known relations.  The generators 
$ g_{s}$,  $s \in S_n$, are invertible.

 	The morphism $R_n$ defined in the beginning of this subsection has a counterpart in the setting of Hecke algebras,    namely the following morphism of algebras (where we write carefully $e_w$ for the basis elements of $H \tilde A_{n-1}(q)$, to be reminded of the possible lack of injectivity):  
	
		\begin{equation}\label{defRn}
\begin{aligned}
					  HR_n: H\tilde{A}_{n-1} (q)  &\longrightarrow   H\tilde{A}_{n} (q) \\
					e_{\sigma_i} &\longmapsto  g_{\sigma_i}  ~~~ ~~~ ~~~ \text{for }  1\leq i\leq n-1  \\
					e_{a_n} &\longmapsto  g_{\sigma_n} g_{a_{n+1}}g_{\sigma_n}^{-1}  .
\end{aligned}		
\end{equation}	 
 
 We have shown in \cite[Proposition 4.3.3]{Sadek_Thesis} that this homomorphism  is injective for $K=\mathbb Z[q, q^{-1}]$ where $q$ is an indeterminate. With a general $K$ as above, we can obtain injectivity using the following technical but crucial result, an immediate consequence of  Theorem \ref{towerandcanonical}  (see \cite[Proposition 3.1]{Sadek_2019}):    

\begin{proposition}\label{coroBC}
	Let $w$ be any element in $W(\tilde{A}_{n-1 } )$, then there exist  $A_w \in  q^\mathds Z$ and  elements  $ \lambda_{x} \in K$  such that 

$$
					HR_n(e_{w}) =A_w \   g_{R_n(w)}+ \sum\limits_{\begin{smallmatrix} x\in W(\tilde{A}_{n}),    \cr  l(x)<l(R_n(w)) \cr L(x)\le L(w) \end{smallmatrix}} \lambda_{x}g_{x}, 
$$
\end{proposition}				  
 
With this, the proof of \cite[Theorem 3.2]{Sadek_2019} applies, we obtain:

\begin{corollary}\label{HeckeA} Let $K$ be a ring and $q$ be invertible in $K$. 
The tower of affine Hecke  algebras: 

$$  H\tilde{A}_{1}(q)  \stackrel{HR_{2}}{\longrightarrow}  H\tilde{A}_{2}(q) \stackrel{HR_{3}} {\longrightarrow}  \cdots  H\tilde{A}_{n-1} (q)\stackrel{HR_{n}} {\longrightarrow}  H\tilde{A}_{n}(q)\longrightarrow  \cdots $$

\medskip\noindent
is a  tower of faithful arrows. 	\\  

\end{corollary}

 
\renewcommand{\refname}{REFERENCES}


\begin{thebibliography}{}

\bibitem{Sadek_Thesis} S. Al Harbat. On the affine braid group, affine Temperley-Lieb algebra and Markov trace. PH.D
Thesis, Universit\'e Paris-Diderot-Paris 7, 2013. \label{Sadek_Thesis}

\bibitem{Sadek_2013_2} S. Al Harbat. Markov trace on a tower of affine Temperley-Lieb algebras of type $\tilde{A }$. 
  J.   Knot Theory Ramifications   24 (9) (2015) 1--28. 
 

\bibitem{Sadek_2016}  S. Al Harbat.   
    Tower of fully commutative elements of type $\tilde A$ and applications.  J. Algebra 465 (2016), 111--136.  
 
 \bibitem{Sadek_2019}  S. Al Harbat.   
    On fully commutative elements of type  $\tilde B$ and $\tilde D$.   
J. Algebra 530 (2019), 1--33. 

\bibitem{BB}  
A. Bj\"orner, F.  Brenti. 
Combinatorics of Coxeter groups.
Graduate Texts in Mathematics, 231. Springer, New York, 2005.

\bibitem{Bourbaki_1981} N. Bourbaki. Groupes et alg\`ebres de Lie, chapitres 4, 5, 6. Masson, 1981. \label{Bourbaki_1981}

\bibitem{BrinkHowlett} B. Brink and R. Howlett. A finiteness property and an automatic structure for Coxeter groups. Math. Ann. 296 (1993), 179--190.

\bibitem{Casselman}   W. Casselman.  
Machine calculations in Weyl groups.    
Invent. Math. 116 (1994), no. 1-3, 95--108. 

\bibitem{duCloux_X} F. du Cloux.  Un algorithme de forme normale pour les groupes de Coxeter. Preprint, Centre de Mathématiques de l'Ecole Polytechnique, 1990. 

\bibitem{duCloux_transducer} F. du Cloux.  A transducer approach to Coxeter groups, J. Symb. Computation 27 (1999), 311--324. 

\bibitem{Lib08}  N. Libedinsky.  Sur la cat\'egorie des bimodules de Soergel. 
J. Algebra 320 (2008), 2675--2694.  

\bibitem{Lusztig}
G. Lusztig. 
Hecke algebras with unequal parameters.
CRM Monograph Series, 18. American Mathematical Society, Providence, RI, 2003.   

\bibitem{St} J. R. Stembridge. Some combinatorial aspects of reduced words in finite Coxeter groups. Transactions A.M.S. 349 (4)  1997, 1285--1332.\label{St}

\bibitem{Yilmaz} E. Yilmaz, C. \"Ozel, U. Ustao\u{g}lu.  Gr\"obner-Shirshov basis and reduced words for affine Weyl group $\tilde A_n$. J. Algebra Appl. 13 (6) (2014) 1450005 (18 pages) 
DOI: 10.1142/S0219498814500054 

\end{thebibliography}
\end{document}